\documentclass[a4paper]{article}
\usepackage{etex}
\usepackage{enumerate}
\usepackage[english]{babel}
\usepackage{graphicx}
\usepackage{multicol}
\usepackage[bottom]{footmisc}
\usepackage[export]{adjustbox}
\usepackage{functan}
\usepackage[latin1]{inputenc}
\usepackage{cancel}
\usepackage{calligra}
\usepackage{colortbl}
\usepackage{multirow}
\usepackage{enumerate}
\usepackage{varioref}
\usepackage{booktabs}
\usepackage{amscd}
\usepackage{color}
\usepackage{colortbl}
\usepackage{pstricks}
\usepackage{pst-all}
\usepackage{mparhack}
\usepackage{amssymb}
\usepackage{amsmath}
\usepackage{dsfont}
\usepackage{mathrsfs}
\usepackage{amsfonts}
\usepackage{amssymb}
\usepackage[mathcal]{euscript}
\usepackage[all,cmtip]{xy}
\usepackage{amssymb}
\usepackage{verbatim}
\usepackage{amsthm}

\theoremstyle{definition}
\newtheorem{definition}{Definition}[section]

\newtheorem{remark}[definition]{Remark}

\theoremstyle{plain}

\newtheorem{theorem}[definition]{Theorem}
\newtheorem{proposition}[definition]{Proposition}
\newtheorem{lemma}[definition]{Lemma}
\newtheorem{corollary}[definition]{Corollary}
\numberwithin{equation}{section}

\def \alt96 {`}
\def \RN {\mathds{R}^N}
\def \R {\mathds{R}}
\def \XX {\mathcal{X}}

\def \dsy {\displaystyle}

\def \loc {\mathrm{loc}}

\def \N {\mathds{N}}

\def \G {\mathbb{G}}
\def \LieG {\mathrm{Lie}(\G)}

\def \e {\varepsilon }
\def \dela {{\delta_\lambda}}

\def \LL {{\mathcal{L}}}

\def \d {\mathrm{d}}
\def \de {\partial}

\def \longto {\longrightarrow}

\def \LL {\mathcal{L}}

\def \supH {\overline{\mathcal{L}}}
\def \subH {\underline{\mathcal{L}}}

\def \subHb {\underline{\mathcal{L}}_b}

\begin{document}
 \author{Stefano Biagi and Ermanno Lanconelli}
 \title{Large sets at infinity and Maximum Principle \\ on unbounded domains for 
 a class of \\ sub-elliptic operators}
 \maketitle
 \begin{abstract}
  \noindent Maximum Principles on unbounded domains play a crucial r\^ole in several problems
  related to linear second-order PDEs of elliptic and parabolic type.
  In this paper we consider 
  a class of sub-elliptic operators $\LL$ in $\RN$ and we   
  establish some criteria for an unbounded open set
  to be a Maximum Principle set for $\LL$. 
  We extend some classical results related to the Laplacian
  (by Deny, Hayman and Kennedy) and to the sub-Laplacians on stratified Lie groups
  (by Bonfiglioli and the second-named author).
 \end{abstract}
 \section{Introduction and main results}  \label{sec:mainassnot}
  It is quite well-known that maximum principles on unbounded 
  domains play a crucial r\^ole in looking
  for symmetry properties of solutions to semilinear Poisson-type
  equations, by using the celebrated
  moving planes or sliding method:
  see, e.g., \cite{BN, BCN, BHM} for the Euclidean setting and
  see \cite{BLanco, BP} for the Heisenberg 
  group setting.
  
  In the present paper we extend to a wide class of subelliptic PDEs
  some maximum principles in unbounded domains 
  holding true for the Euclidean Laplace operator
  (by Deny, Hayman and Kennedy) and for the sub-Laplace
  operators on stratified Lie groups
  (by Bonfiglioli and the second-named author). \bigskip

  \noindent To be more prices, 
  throughout the sequel we shall be concerned with second-order linear
   partial differential operators (PDOs, in the sequel) of the form
   \begin{equation} \label{eq:PDOformintro}
   \begin{split}
   \LL = \frac{1}{V(x)}\,\sum_{i = 1}^N\frac{\de }{\de x_i}
   \bigg(\sum_{j = 1}^NV(x)\,a_{i,j}(x)\,\frac{\de }{\de x_j}\bigg)
   = \frac{1}{V(x)}\,\mathrm{div}\big(V(x)\,A(x)\cdot\nabla\big).
   \end{split}
  \end{equation}
  We shall always assume, without any further comment, 
  that the following \emph{stru\-ctu\-ral 
  assumptions} are satisfied:
  \begin{itemize}
   \item[(H1):] $V,a_{i,j}\in C^\infty(\RN,\R)$ for all $i,j$
   and $V > 0$ on the whole of $\RN$;
   \item[(H2):] the matrix $A(x) = \big(
    a_{i,j}(x)
   \big)_{i,j}$ 
   is symmetric and positive semi-definite for every $x\in\RN$. Furthermore, 
   $$\text{$\mathrm{trace}(A(x)) > 0$ for every $x\in\RN$};$$
   \item[(H3):] there exists a real $\e > 0$ such that
   both $\LL$ and $\LL_\e := \LL - \e$ are $C^\infty$-hypo\-el\-lip\-tic
   in every open subset
   \footnote{We remind that a linear PDO $P$ with smooth coefficients
   is $C^\infty$-hypoelliptic in an open set $\Omega\subseteq\RN$ if any 
   distributional solution to $Pu = f$ is smooth
   in $\Omega$ whenever $f$ is smooth.} of $\RN$. 
  \end{itemize}
  Under these assumptions, a satisfactory Potential Theory for $\LL$ can be constructed
  (see, e.g., 
  \cite{BattagliaBiagi, BBB}).
  In this theory, the ``harmonic'' functions are the \emph{$\LL$-harmonic functions}, 
  that is, the (smooth) solutions to 
  $$\LL u = 0$$
  on some open subset of $\RN$. The corresponding \emph{$\LL$-subharmonic functions}
   are the
  upper semi-continuous (u.s.c., for short) 
  functions $u:\Omega\to [-\infty,\infty)$  (where $\Omega$ is an open subset of 
  $\RN$) such that 
  \begin{itemize}
  \item[(i)] $\{x\in\Omega: u(x) > -\infty\}$ is dense in $\Omega$;
  
  \item[(ii)] for every bounded open set $V\subseteq\overline{V}\subseteq\Omega$ 
  and for every
  function $h$ $\LL$-harmonic in $V$ and continuous up to $\de V$ such that
  $u_{\big|\de V}\leq h_{\big|\de V}$, one has $u\leq h$ in $V$.
  \end{itemize}
  {As a consequence of the (strong) Harnack inequality for $\LL$
  proved in \cite{BBB} (and of the fact that $h\equiv 1$
  is $\LL$-harmonic), the following Maximum Principle for $\LL$-subharmonic
  functions holds true
  (see Theorem \ref{thm.minimumsupH} in the Appendix):}
  \medskip
  
  Let $\Omega\subseteq\RN$ be open and \emph{bounded} and 
  let $u\in\subH(\Omega)$. Then
  \begin{equation} \label{eq.WMPintro}
   \text{$\limsup_{x\to \xi}u(x) 
   \leq 0$\,\,for every $\xi\in\de\Omega$} \quad \Longrightarrow
   \quad \text{$u\leq 0$ in $\Omega$}.
  \end{equation}   
  Here and in what follows, we adopt the 
  subsequent notations: \medskip
  
  \noindent -\,\,$\subH(\Omega)$ denotes the cone of the $\LL$-subharmonic
  functions in the open set $\Omega\subseteq\RN$; \vspace*{0.12cm}
  
  \noindent -\,\,$\subHb(\Omega)$ denotes
  the cone of the bounded above $\LL$-subharmonic 
  functions in $\Omega$; \vspace*{0.12cm} 
  
  \noindent -\,\,$\mathcal{L}(\Omega)$ denotes linear space
  of the $\LL$-harmonic functions in $\Omega$; \vspace*{0.12cm}
  
  \noindent -\,\,$\supH(\Omega)$ denotes the cone
  of the $\LL$-superharmonic functions in $\Omega$
  (by definition, a function $u$ is \emph{$\LL$-superharmonic} 
  (in $\Omega$) if $-u$ is $\LL$-subharmonic
  in the same set). \medskip
  
  {A simple yet remarkable consequence of 
  \eqref{eq.WMPintro} is the fact that
  a function $u$ in $C^2(\Omega,\R)$ is 
  $\LL$-subharmonic in $\Omega$ \emph{if and only if}
  $\LL u \geq 0$ on $\Omega$ (see, e.g., \cite{BattagliaBiagi}).} \bigskip
  
  \noindent Obviously, we cannot expect that the previous Maximum Principle holds 
  true if $\Omega$ is
  \emph{not bounded}, and if we do not assume in \eqref{eq.WMPintro} some extra 
  conditions on the function
  $u$; the main aim of this paper is to provide conditions on an
   \emph{unbounded} open set $\Omega$ ensuring
  \eqref{eq.WMPintro} for every \emph{bounded above} $\LL$-subharmonic function in 
  $\Omega$. \medskip

  \noindent To present our main results, it is convenient to fix the following definition.
    \begin{definition} \label{def:MPsetL}
   Let $\Omega\subseteq\RN$ be open. We say that $\Omega$
   is a \emph{maximum principle set} (MP set, in short) for $\LL$
   if it satisfies the following property:
   \begin{equation} \label{eq.MPsetL}
    \begin{cases}
    u\in\subHb(\Omega) \\[0.1cm]
    \dsy \limsup_{x\to\xi}u(x)\leq 0 & \text{for every $\xi\in\de\Omega$}
    \end{cases} \qquad \Longrightarrow \qquad
    u \leq 0 \quad \text{in $\Omega$}.
   \end{equation}
  \end{definition}
  \noindent We point out that any u.s.c.\,function 
  $u:\Omega\to[-\infty,\infty)$ satisfying the boundary
  condition in \eqref{eq.MPsetL} is bounded above if $\Omega$ is bounded. Then, by
  the previously recalled Maximum Principle, \emph{every bounded} open set
  is a MP set for $\LL$ and in \eqref{def:MPsetL} we can 
  replace $\subHb(\Omega)$ with $\subH(\Omega)$.
  \medskip
  
  As we shall see, the notion of maximum principle set
  (for $\LL$) is closely related to the one of $\LL$-largeness
  at infinity, defined as follows.
 \begin{definition} \label{def.LLthinness}
	We say that a subset $F$ of $\RN$ is \emph{$\LL$-thin at infinity}
	if it is possible to find a function $u\in\subHb(\RN)$ such that
	\footnote{If $F$ is bounded, we agree to let 
	$\dsy \limsup_{
	 \begin{subarray}{c}
	  x\to\infty \\
	  x\in F
	  \end{subarray}}u(x) = \infty$.}
	\begin{equation} \label{eq.LLthinness}
	 \limsup_{
	 \begin{subarray}{c}
	  x\to\infty \\
	  x\in F
	  \end{subarray}
	 }u(x) < \limsup_{
	 \begin{subarray}{c}
	  x\to\infty \\
	  x\in \RN
	  \end{subarray}
	 }u(x).
	\end{equation}
    If $F\subseteq\RN$ is \emph{not} $\LL$-thin at infinity, we shall say that
   $F$ is \emph{$\LL$-la\-rge at in\-fi\-ni\-ty}. Explicitly,
  $F$ is $\LL$-large at infinity if and only if
  $$\limsup_{
	 \begin{subarray}{c}
	  x\to\infty \\
	  x\in F
	  \end{subarray}
	 }u(x) = \limsup_{
	 \begin{subarray}{c}
	  x\to\infty \\
	  x\in \RN
	  \end{subarray}
	 }u(x) \quad
	 \text{for every $u\in \subHb(\RN)$}.$$ 
  \end{definition}
  Here is our first basic result.
    \begin{theorem} \label{thm.LLthinMPset}
	An open set $\Omega\subseteq\RN$ is a maximum principle set 
	for $\LL$ \emph{if and only if}
	its complement $\RN\setminus\Omega$ \emph{is $\LL$-large at infinity}.
  \end{theorem}
 The proof of this theorem will be given in Section \ref{sec:thinsets}. 
 In Section \ref{sec:LLthinpset}, assuming that $\LL$ is endowed 
 with a global fundamental solution
 $$(x,y)\mapsto \Gamma(x;y),$$
 smooth out of the diagonal of $\RN\times\RN$ and satisfying 
 suitable structural conditions
 (see, precisely, assumptions (FS), (G) and (L)), we 
 shall provide a geometrical 
 criterion for a set to be $\LL$-large at infinity. This criterion 
 involves the \emph{superlevel sets} of $\Gamma$, which
 shall be called \emph{$\Gamma$-balls}: more precisely, 
 for every $x\in\RN$ and every $r > 0$, the 
 \emph{$\Gamma$-ball with center at $x$
 and radius $r$} is the set
 $$\Omega(x,r) := \bigg\{y\in\RN\,:\,\Gamma(x;y) > \frac{1}{r}\bigg\}.$$
 From our structural assumptions on $\Gamma$ it easily follows that the function
 $$\gamma(x,y) := \begin{cases}
  0, & \text{if $x = y$}, \\
  1/\Gamma(x;y), & \text{if $x \neq y$},
  \end{cases}
  $$
  is a pseudo-metric in $\RN$ and that 
  $\Omega(x,r)$ is actually the metric $\gamma$-ball
  centered at $x$ and with radius $r$, that is,
  $$\Omega(x,r) = \{y\in\RN:\gamma(x,y) < r\}.$$
  \noindent With the $\Gamma$-balls at hand, we can introduce
  the
  definition of $p_\LL$-unbounded set.
    \begin{definition} \label{def.pset}
   Let $F\subseteq\RN$ be any set and let $p\in(1,\infty)$. We say that $F$ 
   is \emph{$p_\LL$-bounded} 
   if there exists a countable family 
   $\mathcal{F} = \big\{\Omega(x_n,r_n)\big\}_{n\in J}$ such that
   \begin{itemize}
    \item[(a)] $F\subseteq \bigcup_{n\in J}\Omega(x_n,r_n)$;
    \item[(b)] $\dsy\sum_{n \in J}\big(\Gamma(0;x_n)\,r_n\big)^p 
    = \sum_{n \in J}\big(r_n/\gamma(0,x_n)\big)^p < \infty$.
   \end{itemize}
 If $F\subseteq\RN$ is not $p_\LL$-bounded, we shall say that
 $F$ is \emph{$p_\LL$-unbounded}.
  \end{definition}
  Then we have the following result.
    \begin{theorem} \label{thm.mainpsetthin}
   Let $F\subseteq\RN$ be any (non-void) set. If 
   there exists some $p \in (1,\infty)$ such that
   $F$ \emph{is $p_\LL$-unbounded}, then $F$ is $\LL$-large at infinity.
  \end{theorem}
  The proof of this theorem rests on the following result, 
  which is of independent interest: 
  it shows a deep property of the bounded above $\LL$-subharmonic
  functions.
   \begin{theorem} \label{lem.mainpsetLLthin}
   Let $u\in\subHb(\RN)$ and let $p\in(1,\infty)$ be arbitrarily fixed. 
   Then, it is possible to construct
   a $p_\LL$-bounded set $F\subseteq\RN$ such that
   \begin{equation} \label{eq.limitoutF}
    \lim_{\begin{subarray}{c}
   x\to\infty \\
   x\notin F
   \end{subarray}}u(x) = \sup_{\RN}u.
   \end{equation}
  \end{theorem}
  {In view of
   the above Theorem \ref{thm.mainpsetthin},
   it seems natural to look for some ``simple''
   criteria allowing to establish if a set $F\subseteq\RN$ is
   $p_\LL$-unbounded (for some $p > 1$).}
  {In Section \ref{sec.suffcondCone}, assuming that 
  the $\Gamma$-balls satisfy a kind of \emph{doubling
  and reverse doubling} condition (see, precisely, 
  assumption (D)), we shall obtain
  such a criterion via the notion of $\Gamma$-cone, which we now introduce.}
  \begin{definition} \label{defi.Gammacone}
   Let $K\subseteq\RN$ be any set. We say that
   $F$ is \emph{$\Gamma$-cone} if it contains
   a countable family $\mathcal{F} = \{\Omega(z_j,R_j\}_{j\in J}$ of
   $\Gamma$-balls such that \medskip
   
   (i)\,\,$\|z_j\|\to\infty$ as $j\to\infty$; \medskip
   
   (ii)\,\,$\displaystyle\liminf_{j\to\infty}\frac{R_j}{\gamma(0,z_j)} > 0$.
  \end{definition}
  Then, the following theorem holds true.
  \begin{theorem} \label{thm.GammaconeIntro}
    Let $F\subseteq\RN$ and let us assume that there exists a $\Gamma$-cone
    $K\subseteq F$. Then, it is possible to find a real $p > 1$ such that
    $F$ is $p_\LL$-unbounded.
  \end{theorem}  
  Gathering together Theorems \ref{thm.LLthinMPset}, 
  \ref{thm.mainpsetthin} and \ref{thm.GammaconeIntro}
  we obtain the following result, in which all
  the hypotheses (H1)-to-(H3), (FS), (G), (L) and (D) are assumed.
  \begin{theorem}  \label{thm.mainSummarizeIntro}
   The open set $\Omega\subseteq\RN$ is a maximum principle 
   set for $\LL$ if one
   of the following (sufficient) conditions is satisfied:
   \begin{itemize}
    \item[\emph{(i)}] $\RN\setminus\Omega$ is $\LL$-large at infinity 
    \emph{(}this condition is also necessary\emph{)};
    \item[\emph{(ii)}] $\RN\setminus\Omega$ is 
    $p_\LL$-unbounded \emph{(}for a suitable $p > 1$\emph{)};
    \item[\emph{(iii)}] $\RN\setminus\Omega$ contains a $\Gamma$-cone.
   \end{itemize}
  \end{theorem}
  \begin{proof}
   (i)\,\,This is precisely the statement of Theorem \ref{thm.LLthinMPset}. \medskip
   
   (ii)\,\,If $\RN\setminus\Omega$ is $p_\LL$-unbounded (for some $p > 1$),
   we know from Theorem \ref{thm.mainpsetthin} that $\RN\setminus\Omega$
   is $\LL$-large at infinity; thus, by 
   (i), $\Omega$ is a MP set for $\LL$. \medskip
   
   (iii)\,\,If $\RN\setminus\Omega$ contains a $\Gamma$-cone $K$, we know from
   Theorem \ref{thm.GammaconeIntro} that there exists a real $p > 1$ such that
   $\RN\setminus\Omega$ is $p_\LL$-unbounded; thus, by (ii),
   we conclude that $\Omega$ is a maximum principle set for $\LL$.
   This ends the proof.
  \end{proof}
   As a consequence of Theorem \ref{thm.mainSummarizeIntro}
   we easily obtain the following result.
  \begin{corollary} \label{cor.WMPIntro}
   Let $\Omega\subseteq\RN$ be an open set
   satisfying one of conditions \emph{(i)-to-(iii)} in
   Theorem \ref{thm.mainSummarizeIntro}. Moreover,
   let
   $f:\Omega\times\R\to\R$ be such that 
   \begin{equation} \label{eq.assumptionf}
    f(x,z)\leq 0 \qquad \text{for every $x\in\Omega$ and $z\geq 0$.}
   \end{equation}
   If $u\in C^2(\Omega,\R)$ is bounded above and satisfies
   \begin{equation} \label{eq.systemLLucu}
    \begin{cases}
     \LL u + f(x,u) \geq 0 & \text{in $\Omega$}, \\
     \dsy\limsup_{x\to y}u(x)\leq 0 & \text{for every $y\in\de\Omega$},
	\end{cases}
   \end{equation}
   then $u\leq 0$ throughout $\Omega$.
  \end{corollary}
  \begin{proof}
  We argue by contradiction and we assume the existence
  of some point $x_0\in\Omega$ such that $u(x_0) > 0$.
  We then consider the following set
  \begin{equation} \label{eq.Omegaplus}
   \Omega^+ := \{x\in\Omega:u(x) > 0\} \neq \varnothing.
   \end{equation}
  By combining \eqref{eq.assumptionf}
  with \eqref{eq.systemLLucu} we infer that, on $\Omega^+$, we have
  $\LL u \geq  -f(x, u) \geq 0$; as a con\-se\-quence,
  $u\in\subH(\Omega^+)$. On the other hand, by the boundary
  condition in \eqref{eq.systemLLucu} and the fact that
  $u \equiv 0$ on $\de\Omega^+\cap\Omega$, it is readily seen that
  $$\limsup_{x\to y}u(x)\leq 0 \quad \text{for every $y\in\de\Omega^+$}.$$
  From this, by arguing exactly as in Lemma \ref{lem.gluingsubHb}, 
  we infer that the function 
  $$v:\Omega\to\R, \qquad v(x)=\max\{u(x),0\}$$ is $\LL$-subharmonic
  in $\Omega$; furthermore, since $u$ is bounded above
  in $\Omega$, the same is true of $v$. 
  Taking into account that, by assumption, $\Omega$
  is a MP-set for $\LL$, we conclude that $v\leq 0$, whence
  $u\leq 0$, but this is in contradiction with
  \eqref{eq.Omegaplus}.
  \end{proof}

  \medskip
  \noindent Finally, in Section \ref{sec.homoperator} we shall prove that the H\"ormander's
  operators \emph{sums of squares of homogeneous vector fields} satisfy all the hypotheses
  of the above Theorem \ref{thm.mainSummarizeIntro}.
  These operators, precisely, are defined as follows. \vspace*{0.1cm}
  
  Let $\mathcal{X} = \{X_1,\ldots,X_m\}$ be a family of linearly 
  independent smooth vector fields
  on Euclidean space $\RN$, with $N\geq 3$, satisfying the following properties:
  \begin{itemize}
   \item[(I)] $X_1,\ldots,X_m$ are $\dela$-homogeneous of degree $1$ with respect to
   a family of non-isotropic dilations $\{\dela\}_{\lambda > 0}$ of the following type
   $$\dela:\RN\to\RN, \qquad \dela(x) = \big(\lambda^{\sigma_1}x_1,\ldots,
   \lambda^{\sigma_N} x_N\big),$$
   where $1 = \sigma_1\leq \ldots\leq \sigma_N$ are positive integers;
   \item[(II)] $X_1,\ldots,X_m$ satisfy the H\"ormander rank condition, i.e.,
   $$\mathrm{dim}\bigg\{X(x): X\in \mathrm{Lie}\big\{X_1,\ldots,X_m\}\bigg\} = N \quad
   \text{for every $x\in\RN$}.$$
  \end{itemize}
  Then, the second-order linear operator $\LL$ defined by
  $$\LL := \sum_{j = 1}^m X_j^2,$$
  will be called a \emph{homogeneous H\"ormander operator}.  \medskip
  
  \noindent We want
  to point out that the class of the homogeneous H\"ormander
  operators contains, as very particular examples,
  the sub-Laplace operators on stratified
  Lie groups and the so-called
  Grushin-type operators on $\RN$ (with $N\geq 3$), together with
  their generalizations: the $\Delta_\lambda$-Laplacians
  (for $\lambda$ smooth) introduced in \cite{KogojLanconelli}.
  
  When $\LL$ is a homogeneous H\"ormander operator, our geometrical
  criteria for $\LL$-largeness at infinity/$p_\LL$-unboundedness 
  take a more explicit form.
  While we directly
  refer to Section \ref{sec.homoperator} for the statement
  and the proof of such ad-hoc criteria, here we only want to present
  the ``homogeneous'' version of the \emph{cone criterion.} \medskip
  
  To this end, it is convenient to fix a definition.
  \begin{definition} \label{defi.delacone}
   Let $C\subseteq\RN$ be any set. We say that
   $C$ is a \emph{non-degenerate $\dela$-cone}
   if it satisfies the following properties:
   \begin{itemize}
    \item[(i)] $\mathrm{int}(F) \neq \varnothing$;
    \item[(ii)] there exists $\lambda_0 > 0$ such that
    $\dela(C)\subseteq C$ for every $\lambda\geq \lambda_0$.
   \end{itemize}
   Here, $\{\delta_\lambda\}_{\lambda > 0}$ denotes the family
   of (non-isotropic) dilations associated with
   the vector fields $X_1,\ldots,X_m$ and appearing in the above 
   assumption
   (I).
  \end{definition}
  Then we have the following
  result, which will be proved is Section \ref{sec.homoperator}.
  \begin{proposition} \label{prop.delaconeIntro}
   If $F\subseteq\RN$ contains a non-degenerate
   $\dela$-cone, then there exists
   $p > 1$ such that $F$ is $p_\LL$-unbounded.
   \emph{(}in the sense of Definition \ref{def.pset}\emph{)}.
  \end{proposition}
  It can be easily proved that
  every half-space of $\RN$ contains a non-degenerate $\dela$-cone
  (see again Section \ref{sec.homoperator});
  as a consequence, by combining Proposition \ref{prop.delaconeIntro} 
  with Theorem \ref{thm.mainSummarizeIntro}, we readily obtain
  the subsequent result.
  \begin{theorem} \label{thm.delaconeIntro}
   Let $\LL$ be a homogeneous
   H\"ormander operator in $\RN$ \emph{(}with $N\geq 3$\emph{)} and let
   $\Omega\subseteq\RN$ be an open
   set satisfying one of the following conditions:
   \begin{itemize}
    \item[\emph{(i)}] $\RN\setminus\Omega$ contains a non-degenerate $\dela$-cone;
    \item[\emph{(ii)}] $\Omega$ is contained in a half-space
    \footnote{Note that this is equivalent to say that $\RN\setminus\Omega$
    \emph{contains} a half-space.} of $\RN$.
   \end{itemize}
   Then $\Omega$ is a maximum principle for $\LL$.
  \end{theorem}
  \begin{proof}
   (i)\,\,If $\RN\setminus\Omega$ contains a non-degenerate $\dela$-cone, it follows
   from Proposition \ref{prop.delaconeIntro} that $\RN\setminus\Omega$
   is
   $p_\LL$-unbounded (for some $p > 1$); as a consequence,
   Theorem \ref{thm.mainSummarizeIntro}-(ii) allows us to conclude
   that $\Omega$ is a maximum principle set for $\LL$. \medskip
   
   (ii)\,\,If $\Omega$ is contained in a half-space $H$, then
   $\RN\setminus\Omega$ contains the half-space
   $H' = \RN\setminus H$; since $H'$
   contains a non-degenerate $\dela$-cone
   (see Remark \ref{rem.halfspaceCone}), we conclude from (i) that
   $\Omega$ is a maximum principle set for $\LL$.
  \end{proof}
  {We point out that, in order to prove that
  any homogeneous H\"ormander operator $\LL = \sum_{j = 1}^mX_j^2$
  satisfies all the hypotheses
  (H1)-to-(H3), (FS), (G), (L) and (D), we make
  crucial use of \emph{global estimates} for two objects
  associated with $\LL$: its global fundamental
  solution (see Theorem \ref{thm.globalSanchez})
  and the measure of the balls in the Carnot-Carath\'{e}odory metric associated with
  $X_1,\ldots,X_m$ (see Theorem \ref{main.ThmNSW}).} \medskip
  
  When $\LL$ is the classical Laplacian or the sub-Laplacian on a stratified
  Lie group, the maximum principle in Corollary 
  \ref{cor.WMPIntro} was proved in \cite{BCN} and in \cite{BonfiLancoMP}, respectively.
  This last paper contains a version of Theorems \ref{thm.LLthinMPset},
  \ref{thm.mainpsetthin}, \ref{lem.mainpsetLLthin}
  and \ref{thm.delaconeIntro}
  for the sub-Laplacians setting.
  We point out that Theorem \ref{lem.mainpsetLLthin}, in the case
  of the classical Laplace operator $\Delta$, is a somehow weaker form
  of a Deny's theorem for $\Delta$-sub\-har\-monic functions (see
  Theorem 3.1 in the monograph
  \cite{HK}).  \bigskip

  \noindent 
  A short description of the contents of our paper
  is now in order.
  \begin{itemize}
   \item In Section \ref{sec:thinsets} we study the relationship
  between the notion of maximum principle set
  for $\LL$ (see Definition \ref{def:MPsetL})
  and the one of $\LL$-thinness (and $\LL$-largeness) at infinity
  (see Definition \ref{def.LLthinness}).
  
  \item In Section \ref{sec:LLthinpset} we make use
  of the notion of $p_\LL$-unboundedness (see Definition \ref{def.pset})
  to give a geometrical sufficient condition 
  for a set to be $\LL$-large at infinity.
  
  \item In Section \ref{sec.suffcondCone}, by means of the notion
  of $\Gamma$-cone (see Definition \ref{defi.Gammacone}, we prove
  that a set is $\LL$-large at infinity if it contains a $\Gamma$-cone.
  
  \item In Section \ref{sec.homoperator} we prove that
  our theory apply to every homogeneous H\"or\-man\-der operator; to this end,
  we show and use some estimates of the fun\-da\-men\-tal solution
  of these operators which are of independent interest.
  
  \item Finally, in the Appendix we remind some basic results
  coming from abstract Potential Theory needed for our study.
  \end{itemize}
 \section{$\LL$-thin sets and Maximum Principle} \label{sec:thinsets}
  {The main aim of this section is to 
  prove Theorem \ref{thm.LLthinMPset} stated in the Introduction.
  To this end, we need to demonstrate a couple
  of preliminary results.} 
  \begin{lemma} \label{lem.gluingsubHb}
   Let $\Omega\subseteq\RN$ be open and let
   $u\in\subH(\Omega)$ be such that
   \begin{equation} \label{eq.conditionlemma}
    \limsup_{x\to \xi}u(x) \leq 0 \quad \text{for every $\xi\in\de\Omega$}.
   \end{equation}
   Then the function $v:\RN\to[-\infty,\infty)$ defined by
   $$v(x) = \begin{cases}
    \max\{u(x),0\}, & \text{if $x\in\Omega$}, \\[0.1cm]
    0, & \text{if $x\in\RN\setminus\Omega$},
   \end{cases}
   $$
   is $\LL$-subharmonic in $\RN$.
  \end{lemma}
  The proof of Lemma \ref{lem.gluingsubHb} requires some basic notions and facts
  coming from Potential Theory; for this reason, we postpone it to the Appendix.
  \begin{lemma} \label{lem.conditionlimsup}
  Let $F\subseteq\RN$ be any (non-void) set and let
  $u\in\subHb(\RN)$. We assume that $u$ is \emph{not constant}
  in $\RN$.  
  Then, the fol\-lo\-wing statements are equivalent:
  \begin{align}
   & \text{\emph{(i)}}\,\,\dsy \limsup_{
	 \begin{subarray}{c}
	  x\to\infty \\
	  x\in F
	  \end{subarray}}u(x) = \limsup_{
	 \begin{subarray}{c}
	  x\to\infty \\
	  x\in \RN
	  \end{subarray}}u(x); \label{eq.conditionlimsupI} \\
	 & \text{\emph{(ii)}}\,\,\dsy\sup_F u = \sup_{\RN}u. \label{eq.conditionlimsupII}
  \end{align}    
   \end{lemma}
   \begin{proof}
   $\text{(i)\,$\Rightarrow$\,(ii).}$\,\,Since, by assumption,
   $u$ is not constant in $\RN$ and the constant functions
   are $\LL$-harmonic, the Minimum Principle
   in Theorem \ref{thm.minimumsupH} implies that
   \begin{equation} \label{eq.inequalitySMP}
    u(x) < \sup_{\RN}u \qquad \text{for every $x\in\RN$}.
   \end{equation}
   As a consequence, it is easy to recognize that
   $$\sup_{\RN}u = \limsup_{
	 \begin{subarray}{c}
	  x\to\infty \\
	  x\in \RN
	  \end{subarray}}u(x).$$
	This last identity allows us to conclude: indeed,
	if \eqref{eq.conditionlimsupI} holds, we have
	\begin{align*}
	 & \limsup_{
	 \begin{subarray}{c}
	  x\to\infty \\
	  x\in F
	  \end{subarray}}u(x) 
	  \leq \sup_{F} u \leq \sup_{\RN} u =
	  \limsup_{
	 \begin{subarray}{c}
	  x\to\infty \\
	  x\in \RN
	  \end{subarray}}u(x) 
	  \stackrel{\eqref{eq.conditionlimsupI}}{=}
	  \limsup_{
	 \begin{subarray}{c}
	  x\to\infty \\
	  x\in F
	  \end{subarray}}u(x),
	\end{align*}
	and thus \eqref{eq.conditionlimsupII} is satisfied, as desired. \medskip
	
	$\text{(ii)\,$\Rightarrow$\,(i).}$\,\,We first claim that, as a consequence of
	\eqref{eq.conditionlimsupII}, one has
	\begin{equation} \label{eq.supisatinfinity}
	 \sup_{F\cap B(0,r)}u\,\,<\,\,\sup_{F\setminus B(0,r)}u \qquad
	 \text{for every $r > 0$}
	\end{equation}
	(here, $B(0,r)$ denotes the Euclidean ball of centre $0$ and radius $r$). 
	
	Indeed,
	let us assume by contradiction that \eqref{eq.supisatinfinity} does not hold
	for some $r_0 > 0$. Since $u$ attains its maximum on any compact subset of $\RN$,
	it is possible to find a suitable point $x_0\in \overline{F\cap B(0,r_0)}$ such that
	$$u(x_0) = \sup_{F\cap B(0,r_0)} \geq \sup_{F\setminus B(0,r_0)} u.$$
	Owing to \eqref{eq.conditionlimsupII}, this implies that
	$$u(x_0) = \sup_F u = \sup_{\RN} u,$$
	which is contradiction with \eqref{eq.inequalitySMP}.
	Now we have established inequality \eqref{eq.supisatinfinity}, we are ready to conclude:
	indeed, by letting $r\to\infty$ in the cited \eqref{eq.supisatinfinity}, we get
	\begin{align*}
	 \sup_{F} u \leq \limsup_{
	 \begin{subarray}{c}
	  x\to\infty \\
	  x\in F
	  \end{subarray}}u(x)
	  \leq \limsup_{
	 \begin{subarray}{c}
	  x\to\infty \\
	  x\in \RN
	  \end{subarray}}u(x)
	  \leq \sup_{\RN} u 
	  \stackrel{\eqref{eq.conditionlimsupII}}{=} \sup_F u,
	\end{align*}
	and this proves that \eqref{eq.conditionlimsupI} is satisfied. This ends the proof.
   \end{proof}
   From Lemma \ref{lem.conditionlimsup} and the definition
   of $\LL$-thin set, we obtain the following result.
   \begin{corollary} \label{cor.FisnotLLthin}
    Let $F\subseteq\RN$ be any set. Then $F$ is $\LL$-thin at infinity
    if and only if it is possible to find a function $u\in\subHb(\RN)$ such that
    $$\sup_{F} u < \sup_{\RN}u.$$
    Conversely, $F$ \emph{is $\LL$-large at infinity}
    if and only if
    $$\sup_{F} u = \sup_{\RN} u \quad \text{for every $u\in\subHb(\RN)$}.$$ 
   \end{corollary}
   Gathering together Lemmas \ref{lem.gluingsubHb} and
   \ref{lem.conditionlimsup}, we are ready to prove
   Theorem \ref{thm.LLthinMPset}.
   \begin{proof} [Proof (of Theorem \ref{thm.LLthinMPset}).]
   We first prove that, if $\Omega$ is a maximum principle set for $\LL$,
   then its complement $\RN\setminus\Omega$ is $\LL$-large at infinity.
   To this end, we choose $u\in\subHb(\Omega)$ (which we may assume
   to be non constant in $\RN$) and we let
   $$u_0 := \sup_{\RN\setminus\Omega} u.$$
   Since $u$ is u.s.c.\,on $\RN$, for every $\xi\in\de\Omega$ we have
   $$\limsup_{x\to\xi}u(x) \leq u(\xi) \leq u_0;$$
   from this, since we are assuming that
   $\Omega$ is a MP set for $\LL$, we obtain
   $$\text{$u \leq u_0$\,\,in $\Omega$, \quad
   whence $\sup_{\RN} u = u_0 = \sup_{\RN\setminus\Omega} u$}.$$
   By Corollary \ref{cor.FisnotLLthin}, 
   we conclude that $\RN\setminus\Omega$ is $\LL$-large at infinity. \medskip
   
   We now assume that $\RN\setminus\Omega$ is $\LL$-large at infinity
   and we prove that $\Omega$ is a maximum principle set for $\LL$.
   To this end, we choose once again a function
   $u\in \subHb(\Omega)$ (which we may assume to be non constant in $\RN$) such that
   $$\limsup_{x\to\xi}u(x)\leq 0 \quad \text{for every $\xi\in\de\Omega$}$$
   and, according to Definition \ref{def:MPsetL}, we prove that $u\leq 0$ in $\Omega$.
   To begin with, owing to Lemma \ref{lem.gluingsubHb}, we see that the function
   $v:\RN\to[-\infty,\infty)$ defined by
   $$v(x) = \begin{cases}
    \max\{u(x),0\}, & \text{if $x\in\Omega$}, \\[0.1cm]
    0, & \text{if $x\in\RN\setminus\Omega$},
   \end{cases}
   $$
   is a $\LL$-subharmonic function in $\RN$ which is also bounded from above
   (as the same is true of $u$); since we are assuming that the set
   $\RN\setminus\Omega$ is $\LL$-large at infinity, we deduce
   from Corollary \ref{cor.FisnotLLthin} that
   $\sup_{\RN} v = \sup_{\RN\setminus\Omega} v = 0$, whence
   $$u(x) \leq \max\{u(x),0\} = v(x) \leq 0 \quad \text{for every $x\in\Omega$}.$$
   This ends the proof.
   \end{proof}
   \section{$\LL$-thinness at infinity and $p$-boundedness} \label{sec:LLthinpset}
  {The aim of this second section is to demonstrate the geometrical criterion
  for $\LL$-largeness at infinity contained in Theorem \ref{thm.mainpsetthin}.
   To this end, as already anticipated in the Introduction, 
   we need to require our PDOs $\LL$ to satisfy some additional
   assumptions, which we now properly introduce.}
   \begin{itemize}
    \item[(FS)] First of all, we assume that $\LL$ is endowed with a \emph{``well-behaved'' global fundamental
    solution}, that is, there exists a function
    $$\Gamma:\mathcal{O} := \big\{(x,y)\in\RN\times\RN: x\neq y\big\}\longto \R$$
    satisfying the following properties: \medskip
   \begin{itemize}
    \item[(a)] $\Gamma\in C^\infty(\mathcal{O},\R)$ and $\Gamma(x,y) > 0$
    for every $x,y\in\mathcal{O}$;
    \item[(b)] $\Gamma$ is symmetric, that is, $\Gamma(x,y) = \Gamma(y,x)$ for every
    $(x,y)\in\mathcal{O}$;
    \item[(c)] for every $x\in\RN$, we have
    $\Gamma(x,\cdot)\in L^1_{\mathrm{loc}}(\RN)$ and
    \begin{equation} \label{eq.deffundsol}
  \int_{\RN}\Gamma(x,y)\,\LL \varphi(y)\,V(y)\,\d y = -\varphi(x),
  \quad \text{$\forall\,\varphi\in C_0^\infty(\RN,\R)$};
   \end{equation}
   \item[(d)] for every $x \in \RN$, $\Gamma(x,\cdot)$
   has a pole at $x$ and it vanishes at infinity, i.e,
   \begin{equation} \label{eq.limitGamma}
   \lim_{y\to x}\Gamma(x,y) = \infty \qquad \text{and} \qquad
   \lim_{\|y\|\to\infty}\Gamma(x,y) = 0.
  \end{equation}
   \end{itemize}
   For the sake of brevity, given $x\in \RN$, in the sequel we set:
    $$\Gamma_x:\RN\setminus\{x\}\longto \R,\quad \Gamma_x(y) := \Gamma(x,y). $$

  \item[(G)] Our second assumption is a sort of ``\emph{geometric condition}''
  which concerns the super-level sets of the fundamental solution $\Gamma$.
  
  More precisely, for every fixed $x\in\RN$ and every $r > 0$, we define
  the \emph{open $\Gamma$-ball} of centre $x$ and radius $r$ in the following way
  $$\Omega(x,r) := \{y\in\RN\setminus\{x\}: \Gamma_x(y) > 1/r\}\cup\{x\};$$
  we then assume the existence of constant $\theta\in (0,1)$ such that
  \begin{equation} \label{eq.geometricGammaball}
   x\notin \Omega(y,r)\qquad \Longrightarrow
   \qquad \Omega(x,\theta r)\cap\Omega(y,\theta r) = \varnothing
  \end{equation}  	
  for every $x,y\in\RN$ and every $r > 0$.
  \item[(L)] Finally, we suppose that the following \emph{Liouville-type theorem} holds
  for $\LL$-harmonic functions: if $u\in\mathcal{L}(\RN)$ is a $\LL$-harmonic function
  which is bounded from above (or from below), then $u$ is constant throughout $\RN$.
  \end{itemize}
 Under our assumptions (FS), (G) and (L), we have the 
 following crucial result.
 \begin{theorem} \label{thm.reprformulaandintegral}
  Let $u\in\subHb(\RN)$ and let $\mu$ be $\LL$-Riesz measure of
  $u$. 
  Then
  \begin{equation} \label{eq.convintegralmu}
   \int_1^\infty\frac{\mu\big(\Omega(0,r)\big)}{r^2}\,\d r < \infty.
  \end{equation}
  Moreover, if $u_0 = \sup_{\RN}u$, we have the representation
  formula
  \begin{equation} \label{eq.reprglobal}
   u(x) = u_0 - \int_{\RN}\Gamma(x,y)\,\d\mu(y), \qquad\text{for every $x\in\RN$}.
  \end{equation}
  \end{theorem}
  {It is proved in \cite{BattagliaBiagi} that, if $\Omega\subseteq\RN$
  is an open set and $u\in \subH(\Omega)$ (not necessarily bounded above), 
  then $u\in L^1_{\loc}(\Omega)$ and
  $\LL u \geq 0$ in the sense of distribution on $\Omega$. Hence,
  the $\LL$-Riesz measure $\mu$ of $u$ is defined by $\mu := \LL u$.}
  \begin{proof}
  {The proof of \eqref{eq.convintegralmu} is analogous to that of    
  \cite[Theorem 9.6.1]{BLUlibro}: it is crucially based on assumption (G) and
  on the mean value formulas 
  for $\LL$ established in \cite{BattagliaBonfiglioli}.}
  
  {As for the proof of
  \eqref{eq.reprglobal}, it can be obtained
  by combining the Liouville-type theorem in 
  assumption (L) with \cite[Remark 5.5]{BattagliaBiagi}
  (see also \cite[Corollary 9.4.8]{BLUlibro}).}
  \end{proof}
  \begin{remark} \label{rem.assumptionGdistance}
 	We point out, for future reference, that the ``geometric condition''
 	in assumption (G) is actually equivalent to requiring that the function
 	$$\RN\times\RN \ni (x,y)\mapsto \gamma(x,y) = \gamma_x(y) := \begin{cases}
 	1/\Gamma(x,y), & \text{if $x\neq y$}, \\
 	0, & \text{if $y = x$},
 	\end{cases}$$ 
 	satisfies a pseudo-triangle inequality,
 	that is, there exists $\mathbf{c} > 1$ such that
 	\begin{equation} \label{eq.pseudogamma}
 	 \gamma(x,y) \leq \mathbf{c}\big(\gamma(x,z)+\gamma(z,y)\big) 
 	\qquad\text{for every $x,y,z\in\RN$}.
 	\end{equation}
 	Indeed, if \eqref{eq.pseudogamma} holds, it is very easy to recognize that
 	assumption (G) is satisfied with $\theta = 1/(2\mathbf{c}) < 1$. On the other hand, if 
 	\eqref{eq.geometricGammaball} holds, one has
 	$$\Omega\big(x,\theta/\Gamma(x,y)\big)\cap\Omega\big(y,\theta/\Gamma(x,y)\big) = \varnothing
 	\quad \text{for every $x\neq y$}.$$
 	From this, we easily obtain the validity of \eqref{eq.pseudogamma} with $\mathbf{c} = 1/\theta > 1$.  	
 \end{remark} 
 \begin{remark} \label{rem.GGammaquasid}
 	By Remark \ref{rem.assumptionGdistance}
 	and the properties of $\Gamma$ listed in assumption (FS), we derive
 	that $\gamma = 1/\Gamma$ is a \emph{quasi-distance} in $\RN$.
 	In fact, we have
 	\begin{itemize}
 	 \item $\gamma \geq 0$ on $\RN\times\RN$ and $\gamma(x,y) = 0$ if and only if $x = y$;
 	 \item $\gamma(x,y) = \gamma(y,x)$ for every $x,y\in\RN$;
 	 \item $\gamma(x,y)\leq\mathbf{c}\big(\gamma(x,z)+\gamma(z,y)\big)$
 	 for every $x,y,z\in\RN$.
 	\end{itemize}
 	Furthermore, for every $x\in\RN$ and every $r > 0$ we have
 	$$\Omega(x,r) = \{y\in\RN:\gamma(x,y) < r\}.$$
  \end{remark}
  {Now we have introduced assumptions (FS), (G) and (L), we proceed to
  the proof of Theorem \ref{thm.mainpsetthin}. 
  To begin with, we list in the next remark some useful pro\-per\-ties
  of $p_\LL$-bounded sets which follow immediately from Definition \ref{def.pset}.}
  \begin{remark} \label{rem.propertiespset}
   (1) If $F\subseteq\RN$ is bounded, then $F$ is $p_\LL$-bounded (for any $p > 1$). \medskip
   
   \noindent 
   (2) If $F_0\subseteq F$ and $F$ is $p$-bounded, then also $F_0$ is $p_\LL$-bounded. \medskip
    
    \noindent (3) If $\{F_j\}_{j = 1}^n$ are $p_\LL$-bounded (for the same $p$), then
    $F = \cup_j F_j$ is $p_\LL$-bounded. \medskip
    
    \noindent (4) If $F_0\subseteq F$ is $p_\LL$-bounded, then $F\setminus F_0$ is
    $p_\LL$-unbounded whenever $F$ is.
  \end{remark}
  {We then turn to demonstrate Theorem \ref{lem.mainpsetLLthin}
  stated in the Introduction: as anticipated,
  this result is
  the key tool for proving Theorem \ref{thm.mainpsetthin}. In its turn, the proof
  of the cited Theorem \ref{lem.mainpsetLLthin} is crucially based
  on the next lemma.}
  \begin{lemma} \label{lem.Cartantype}
   Let $\mu$ be a positive Radon measure on $\RN$ such that
   $\mu_0 = \mu(\RN)$ is \emph{finite}. Moreover, let $p\in(1,\infty)$ be fixed
   and let $h > 0$. Then the set
   $$\{x\in\RN: \Gamma\mu(x) \geq h\} = \bigg\{x\in\RN: \int_{\RN}\Gamma(x,y)\,\d\mu(y) \geq h\bigg\}$$
   can be covered a finite or countable family $\mathcal{F} = \{\overline{\Omega(x_n,r_n)}\}_{n\in J}$
   of closed $\Gamma$-balls satisfying the following property: \emph{there exists a real
   constant $A_p > 0$ such that}
   \begin{equation} \label{eq.Cartancovering}
    \sum_{n\in J}(r_n)^p < A_p\,\bigg(\frac{\mu_0}{h}\bigg)^p.
   \end{equation}
  \end{lemma}
  \begin{proof}
	For every fixed natural $n$, we define
	$$r_n := \frac{\mu_0}{h}\cdot 2^{-2n/(p+1)}$$
	and we choose a maximal family $\mathcal{D}_n$ of disjoint $\Gamma$-balls
	of radius $r_n$
	such that
	$$\mu(B) \geq \frac{\mu_0}{2^n} \quad \text{for every $B\in\mathcal{D}_n$}.$$
	Since, by assumption, $\mu_0 = \mu(\RN) < \infty$ and the $\Gamma$-balls in $\mathcal{D}_n$ are disjoint,
	it is readily seen that $\mathcal{D}_n$ consists of at most $k_n\leq 2^n$ elements; hence, we write
	$$\mathcal{D}_n = \Big\{\Omega_{k,n} = \Omega(x_{k,n},r_n)\,\,:\,\,k = 1,\ldots,k_n\Big\}.$$
	If $\theta$ is the constant appearing in assumption (G), we then define
	$$F := \bigcup_{n = 1}^{\infty}\bigcup_{k = 1}^{k_n}\Omega(x_{k,n},r_n/\theta).$$
	We now observe that, if $x\notin F$, then $\Omega(x,r_n)$ does not intersect any element of
	the family 
	$\mathcal{D}_n$: in fact, since $x\notin \Omega(x_{k,n},r_n/\theta)$, assumption (G) implies that
	$$\Omega(x,r_n)\cap\Omega(x_{k,n},r_n) = \varnothing \quad
	\text{for every $n\in\N$ and every $k\leq k_n$}.$$
	As a consequence, since $\mathcal{D}_n$ is maximal, 
	we infer that $\Omega(x,r_n)\notin\mathcal{D}_n$, whence
	$$\mu\big(\Omega(x,r_n)\big)\leq \frac{\mu_0}{2^n} \quad
	\text{for every $n\in\N$}.$$
	In particular, $\mu(\{x\}) = 0$. For every $x\in F$, we then have
	\begin{align*}
	 \Gamma\mu(x) & = \int_{\RN\setminus\{0\}}\Gamma(x,y)\,\d\mu(y) \\[0.2cm]
	 & = \bigg(\int_{\RN\setminus\overline{\Omega(x,r_1)}} + \sum_{n = 1}^\infty
	 \int_{\Omega(x,r_n)\setminus\overline{\Omega(x,r_{n+1})}}\bigg)\Gamma(x,y)\,\d\mu(y) \\[0.2cm]
	 & \leq \frac{\mu_0}{r_1} + \sum_{n = 1}^\infty \frac{\mu\big(\Omega(x,r_{n})\big)}{r_{n+1}} 
	 \leq \mu_0\bigg(\frac{1}{r_1}+\sum_{n = 1}^\infty\frac{1}{2^n\,r_{n+1}}\bigg) \\[0.2cm]
	 & = \mu_0\,\sum_{n = 1}^\infty 2^{1-n}\,r_{n}^{-1}
	 = 2\,h\,\sum_{n = 1}^\infty \bigg(2^{\frac{1-p}{1+p}}\bigg)^n \\[0.2cm]
	 & = A_p\,h,
	\end{align*}
	where $A_p$ only depends on $p > 1$. We have thus proved that, for every $x\in F$, we have
	$\Gamma\mu(x)\leq A_p\,h$; this obviously implies the inclusion
	$$\{x\in\RN: \Gamma\mu(x) > A_p\,h\} \subseteq F = 
	\bigcup_{n = 1}^{\infty}\bigcup_{k = 1}^{k_n}\Omega(x_{k,n},r_n/\theta).$$
	Furthermore, since $k_n\leq 2^n$, we have
	\begin{align*}
	 & \sum_{n = 1}^\infty\sum_{k = 1}^{k_n}\big(r_n\big)^p
	 \leq \bigg(\frac{\mu_0}{h}\bigg)^p\cdot \sum_{n = 1}^\infty 2^n\,(r_n)^p \\
	 & \qquad\qquad = \bigg(\frac{\mu_0}{h}\bigg)^p\cdot\sum_{n = 1}^\infty\bigg(2^{\frac{1-p}{1+p}}\bigg)^n
	 = A_p\,\bigg(\frac{\mu_0}{h}\bigg)^p.
	\end{align*}
	Since the constant $A_p$ is positive and only depends on $p$, the lemma is proved.
  \end{proof}
  {With Lemma \ref{lem.Cartantype} at hand, 
  we can prove Theorem \ref{lem.mainpsetLLthin}.}
  \begin{proof} [Proof (of Theorem \ref{lem.mainpsetLLthin}).]
  Let $u$ be as in the statement of the theorem. Moreover, 
   let $\theta$ be the constant appearing in assumption (G) and let $n\in\N$ be fixed.
   If $\mu$ is the $\LL$-Riesz measure of $u$ and $u_0 = \sup_{\RN}u$, by Theorem
   \ref{thm.reprformulaandintegral} we have
   $$u(x) = u_0 - \int_{\RN}\Gamma(x,y)\,\d\mu(y) = I_1(x)+I_2(x)+I_3(x),$$
   where we have used the notations 
   \begin{equation*}
   \begin{split}
    & I_1(x) := \int_{\{\Gamma_0(y)\geq\theta^{n-1}\}}\Gamma(x,y)\,\d\mu(y); \\[0.2cm]
    & \qquad\quad I_2(x) := \int_{\{\theta^{n+2}<\Gamma_0(y)<\theta^{n-1}\}}\Gamma(x,y)\,\d\mu(y); \\[0.2cm]
    & \qquad\qquad
    \quad\,\, I_3(x) := \int_{\{\Gamma_0(y)\leq\theta^{n+2}\}}\Gamma(x,y)\,\d\mu(y).
   \end{split}
   \end{equation*}
   We then consider the set $\Omega_n$ defined by
   $$\Omega_n := \{x\in\RN: \theta^{n+1} < \Gamma_0(x) \leq \theta^n\}.$$
   and we proceed by estimating $I_1(x), I_2(x)$ and $I_3(x)$
   when $x\in\Omega_n$. \medskip
   
   \textsc{Estimate of $I_1$}. We first observe that, if $x\in\Omega_n$,
   then $x\notin \Omega(0,\theta^{-n})$; thus, by assumption (G), we have
   $\Omega(x,\theta^{1-n})\cap\Omega(0,\theta^{1-n})=\varnothing$, whence
   $$\Gamma(x,y)\leq \theta^{n-1} \quad \text{for every $y\in\overline{\Omega(0,\theta^{1-n})}$}.$$
   From this, we obtain the following estimate for $I_1(x)$:
   \begin{equation} \label{eq.estimI1}
   \begin{split}
    I_1(x) & \leq \theta^{n-1}\,\mu\big(\overline{\Omega(0,\theta^{1-n})}\big)
    \leq \frac{1}{\theta}\cdot\Big(\theta^n\,\mu\big(\Omega(0,\theta^{-n})\big)\Big) \\[0.2cm]
    & = \frac{1}{\theta}\cdot
    \Big(\mu\big(\Omega(0,\theta^{-n})\big)\,\int_{\theta^{-n}}^\infty\frac{1}{r^2}\,\d r\Big) \\[0.2cm]
    & \leq \frac{1}{\theta}\cdot\int_{\theta^{-n}}^{\infty}\frac{\mu\big(\Omega(0,r)\big)}{r^2}\,\d r.
    \end{split}
    \end{equation}
    We explicitly point out that, as a consequence of Theorem \ref{thm.reprformulaandintegral}, 
    the integral in the far left-hand side of the above inequality is finite. \medskip
    
    \textsc{Estimate of $I_3$}. Let $y\in\RN$ be such that $\Gamma_0(y)\leq \theta^{n+2}$ and let
    $$\rho_y = (\Gamma_0(y))^{-1} > 0.$$
    Since, obviously, $y\notin\Omega(0,\rho_y)$, assumption (G) implies
    that the $\Gamma$-balls 
    $\Omega(0,\theta\rho_y)$ and $\Omega(y,\theta\rho_y)$ are disjoint; on the other hand,
    if $x\in \Omega_n$, one has
    $$\Gamma_0(x) > \theta^{n+1} = \frac{1}{\theta}\cdot\theta^{n+2} \geq \frac{\Gamma_0(y)}{\theta}
    = \big(\theta\rho_y\big)^{-1},$$
    and thus $x\in\Omega(0,\theta\rho_y)$. As a consequence, we derive that
    $x\notin\Omega(y,\theta\rho_y)$, whence
    $$\Gamma(y,x) = \Gamma(x,y) \leq \big(\theta\rho_y\big)^{-1} = \frac{\Gamma(0,y)}{\theta}.$$
    By exploiting this last estimate, we obtain
    \begin{equation*}
    \begin{split}
    I_3(x) & \leq \frac{1}{\theta}\cdot\int_{\{\Gamma_0(y)\leq\theta^{n+2}\}}\Gamma_0(y)\,\d\mu(y) \\[0.2cm]
    & = \frac{1}{\theta}\cdot\int_{\{\Gamma_0(y)\leq\theta^{n+2}\}}\bigg(
    \int_{1/\Gamma_0(y)}^{\infty}\frac{1}{r^2}\,\d r\bigg)\,\d\mu(y) \\[0.2cm]
    & \leq \frac{1}{\theta}\cdot\int_{\theta^{-n-2}}^\infty\frac{\mu\big(\Omega(0,r)\big)}{r^2}\,\d r
    \\[0.2cm]
    & \leq 
    \frac{1}{\theta}\cdot\int_{\theta^{-n}}^{\infty}\frac{\mu\big(\Omega(0,r)\big)}{r^2}\,\d r.
    \end{split}
    \end{equation*}
    
    \textsc{Estimate of $I_2$}. The estimate of $I_2(x)$ (when $x\in \Omega_n$) is the crucial part
    of the proof. To begin with, we fix $p\in(1,\infty)$ and we define
    \begin{equation} \label{eq.defmunetan}
     \begin{split}
      & \mu_n := \mu\big(\{y\in\RN : \theta^{n+2}<\Gamma_0(y)<\theta^{n-1}\}\big); \\[0.2cm]
      & \eta_n := \mu_n\,\theta^n; \\[0.2cm]
      & \e_n := \eta_n^{1-1/p}.
     \end{split}
    \end{equation}
    We claim that series $\sum_{n = 1}^\infty\eta_n$ is convergent.
    In fact, for every $n\in\N$ we have
    \begin{align*}
    \eta_n & \leq \frac{1}{\theta^2}\cdot\int_{\{\theta^{n+2}<\Gamma_0(y)<\theta^{n-1}\}}\Gamma_0(y)\,\d\mu(y) \\[0.2cm]
    & \leq \frac{1}{\theta^2}\cdot\bigg(\int_{\{\Gamma_0(y)\leq \theta^{n-1}\}}
    -\int_{\{\Gamma_0(y)\leq\theta^{n+2}\}}\bigg)\Gamma_0(y)\,\d\mu(y) \\[0.2cm]
    & \leq \frac{1}{\theta^2}\cdot
    \bigg\{\int_{\theta^{1-n}}^{\theta^{-2-n}}\frac{\mu\big(\Omega(0,r)\big)}{r^2}\,\d r  \\
    & \qquad\qquad + \Big(\theta^{n+2}\,\mu\big(\Omega(0,\theta^{-2-n})\big)
    - \theta^{n-1}\,\mu\big(\Omega(0,\theta^{1-n})\big) 
    \Big)\bigg\}.
    \end{align*}
    On the other hand, by arguing as for the estimate of
    $I_1(x)$, we see that
    $$r\,\mu\big(\Omega(0,r)\big) \leq \int_r^\infty\frac{\mu\big(\Omega(0,t)\big)}{t^2}\,\d t \longto 0 \quad
    \text{as $r\to\infty$};$$
    as a consequence, we obtain
    \begin{align*}
     & \sum_{n = 1}^\infty\Big(\theta^{n+2}\,\mu\big(\Omega(0,\theta^{-2-n})\big)
    - \theta^{n-1}\,\mu\big(\Omega(0,\theta^{1-n})\big) 
    \Big) \\
    & \quad = -\Big(\theta^2\,\mu\big(\Omega(0,\theta^{-2})\big)
     + \theta\,\mu\big(\Omega(0,\theta^{-1})\big) + \mu\big(\Omega(0,1)\big)\Big) < 0.
    \end{align*}
    Gathering together all these facts, we conclude that (see Theorem \ref{thm.reprformulaandintegral})
    $$\sum_{n = 1}^\infty \eta_n \leq \frac{1}{\theta^2}\cdot\int_1^\infty
    \frac{\mu\big(\Omega(0,r)\big)}{r^2}\,\d r < \infty,$$
    as claimed. In particular, we have $\eta_n\longto 0$ as $n\to\infty$.
    
    We now observe that, if we consider the Radon measure $\lambda_n$ defined by
    $$\d\lambda_n = \mathbf{1}_{\{\theta^{n+2}<\Gamma_0<\theta^{n-1}\}}\,\d\mu,$$
    we have $\lambda_n(\RN) = \mu_n < \infty$ and, for every $x\in\RN$, we can write
    $$I_2(x) = \int_{\RN}\Gamma(x,y)\,\d\lambda_n(y) = \Gamma\lambda_n(x).$$
    By Lemma \ref{lem.Cartantype}, it is then possible to find a
    family $\mathcal{F}_n = \{\overline{\Omega(x_{k,n},r_{k,n})}\}_{k\in J_n}$
    (with $J_n\subseteq\N$)
   of closed $\Gamma$-balls satisfying the following properties: \medskip
   
   (i)\,\,$\{x\in\Omega_n : I_2(x) < \e_n\} \supseteq \Omega_n\setminus\bigcup_{k \in J_n}
   \overline{\Omega(x_{k,n},r_{k,n})}$; \medskip
   
   (ii)\,\,$\sum_{k\in J_n}(r_{k,n})^p < A_p\,\big(\mu_n/\e_n\big)^p$ for a suitable constant
   $A_p > 0$. \medskip
   
   \noindent As a consequence of property (ii), for every $k\in J_n$ we have
   \begin{equation} \label{eq.touseestim}
    r_{k,n} \leq A_p^{1/p}\,\big(\mu_n/\e_n) \stackrel{\eqref{eq.defmunetan}}{=} 
    (A_p\,\eta_n)^{1/p}\,\theta^{-n};
    \end{equation}
    moreover, by property (i), we can assume that $\Omega(x_{k,n},r_{k,n})\cap \Omega_n \neq \varnothing$
    for every index $k\in J_n$. This implies the existence of $n_0\in\N$ such that
    $$\Gamma_0(x_{k,n}) \leq \theta^{n-2} \quad
    \text{for every $n\geq n_0$ and every $k\in J_n$}.$$
    Indeed, if $z$ is any point in $\Omega(x_{k,n},r_{k,n})\cap \Omega_n\subseteq\Omega_n$, we see that
    $z\notin \Omega(0,\theta^{-n})$ and thus, again by assumption (G),
    we infer that $$\Omega(z,\theta^{1-n})\cap\Omega(0,\theta^{1-n}) = \varnothing.$$
    On the other hand, since $z$ also belongs to $\Omega(x_{k,n},r_{k,n})$, by \eqref{eq.touseestim} one has
    $$\Gamma(x_{k,n},z) = \Gamma_z(x_{k,n}) > \frac{1}{r_{k,n}} \geq \theta^n\,\big(A_p\,\eta_n\big)^{-1/p};$$
    as a consequence, if $n_0\in\N$ is such that
    $(A_p\,\eta_n)^{-1/p} > \theta^{-1}$ for every $n \geq n_0$ (note that
    $\eta_n\to 0$ as $n\to\infty$ and $-1/p < 0$), we derive that
    $\Gamma_z(x_{k,n}) > \theta^{n-1}$, whence $x_{k,n}\in\Omega(z,\theta^{1-n})$, and thus
    $x_{k,n}\notin\Omega(0,\theta^{1-n})$. This implies that
    $$\Gamma_0(x_{k,n})\leq \theta^{n-1} < \theta^{n-2} \quad
    \text{for every $n\geq n_0$ and every $k\in J_n$}.$$
 	By combining this last estimate with the choice of $\e_n$ and property (ii), we get
 	\begin{equation} \label{eq.touseFispset}
 	 \begin{split}
 	 \sum_{n = n_0}^\infty\sum_{k\in J_n}&\big(\Gamma_0(x_{k,n})\,r_{k,n}\big)^p
 	 \leq (A_p\,\theta^{-2p})\,\sum_{n = n_0}^\infty\theta^{pn}\,\bigg(\frac{\mu_n}{\e_n}\bigg)^p
      \\
      & \qquad = (A_p\,\theta^{-2p})\cdot\sum_{n = n_0}^\infty\eta_n < \infty.
 	\end{split}
 	\end{equation}
 	Furthermore, by collecting the estimates for $I_1(x),I_2(x)$ and $I_3(x)$, we obtain
 	\begin{equation} \label{eq.touselimit}
 	 \begin{split}
 	 u_0 - u(x) & \leq \frac{2}{\theta}\cdot\int_{\theta^{-n}}^\infty\frac{\mu\big(\Omega(0,r)\big)}{r^2}\,\d r
 	 + \eta_n^{1-1/p}
 	 \end{split}
 	\end{equation}
 	for every $x\in\Omega_n$ such that $I_2(x) < \e_n$.
 	We finally claim that the set
 	$$F := \bigcup_{n \geq n_0}\bigcup_{k\in J_n}\Omega(x_{k,n},r_{k,n})$$
 	is $p_\LL$-bounded and it satisfies \eqref{eq.limitoutF}. In fact, if we introduce the family 
 	$$\mathcal{F} := \big\{\Omega(x_{k,n},r_{k,n})\big\}_{
 	\begin{subarray}{c}
 	n\geq n_0 \\
 	k\in J_n
 	\end{subarray}},$$
 	we derive from \eqref{eq.touseFispset} that $\mathcal{F}$ (which is obviously
 	a countable cover of $F$)
 	satisfies property (b) in Definition \ref{def.pset}, hence $F$ is $p_\LL$-bounded;
 	moreover, since $\theta^{-n}\to\infty$ and $\eta_n\to 0$ as $n\to\infty$
 	(note that $\theta < 1$ and $p > 1$), for every $\e > 0$
 	it is possible to find $n_\e \geq n_0$ such that (see also \eqref{eq.convintegralmu} in
 	Theorem \ref{thm.reprformulaandintegral})
 	\begin{equation} \label{eq.touselimit2}
 	\frac{2}{\theta}\cdot\int_{\theta^{-n}}^\infty\frac{\mu\big(\Omega(0,r)\big)}{r^2}\,\d r
 	 + \eta_n^{1-1/p} < \e \quad
 	 \text{for every $n\geq n_\e$}.
 	 \end{equation}
 	 On the other hand, for every $x\in\RN\setminus\overline{\Omega(0,\theta^{-n_\e})}$ 
 	 (which is an open neighborhood of $\infty$) non belonging to $F$, 
 	 there exists a (unique) $n \geq n_\e\geq n_0$ such that
 	 $$x\in\Omega_n\setminus\bigcup_{k\in J_n}\Omega(x_{k,n},r_{k,n})
 	 \subseteq\{z\in\RN: I_2(z) < \e_n\};$$
 	 as a consequence, by combining \eqref{eq.touselimit} with \eqref{eq.touselimit2} we conclude that
 	 $$0\leq u_0 - u(x) < \e \quad 
 	 \text{for every $x\in \RN\setminus F$ with $\Gamma_0(x)\leq \theta^{-n_\e}$}.$$
 	 This shows that \eqref{eq.limitoutF} holds true, and the proof is complete.
  \end{proof}
  Now we have established
  Theorem \ref{lem.mainpsetLLthin},
  we can finally prove Theorem \ref{thm.mainpsetthin}.
  \begin{proof} [Proof (of Theorem \ref{thm.mainpsetthin}).]
  We demonstrate the following
  equivalent fact: \emph{if $F\subseteq\RN$ is $\LL$-thin at infinity,
  then $F$ is $p_\LL$-bounded \emph{for any $p > 1$}.}
  \medskip
  
  To this end, we let $F\subseteq\RN$ be $\LL$-thin at infinity and, by contradiction,
  we suppose that $F$ \emph{is} $\overline{p}_\LL$-unbounded for a certain 
  $\overline{p}> 1$. 
   If $u\in\subHb(\RN)$
   is fixed, we infer from Lemma \ref{lem.mainpsetLLthin} the existence
   of a $\overline{p}_\LL$-bounded set $F_0\subseteq\RN$ such that
   \begin{equation} \label{eq.limitF0touse}
    \lim_{\begin{subarray}{c}
   x\to\infty \\
   x\notin F
   \end{subarray}}u(x) = \sup_{\RN}u;
   \end{equation}
   moreover, since $F_0$ is $\overline{p}_\LL$-bounded but $F$ is not, then 
   $F\setminus F_0\subseteq F$ \emph{is $\overline{p}_\LL$-unbounded}. 
   In particular, $F\setminus F_0$ is non-void and unbounded
   (see (2) and (4) in Remark \ref{rem.propertiespset}).
	By combining this last fact with \eqref{eq.limitF0touse}, we then obtain
	$$
    \lim_{\begin{subarray}{c}
   x\to\infty \\
   x\in F\setminus F_0
   \end{subarray}}u(x) = \sup_{\RN}u, 
	$$	
	which obviously implies that
	$$\sup_F u \geq \limsup_{
	 \begin{subarray}{c}
	  x\to\infty \\
	  x\in F
	  \end{subarray}}u(x)
	  \geq \lim_{\begin{subarray}{c}
	   x\to\infty \\
   	x\in F\setminus F_0
   	\end{subarray}}u(x) = \sup_{\RN} u.$$
   	Owing to Corollary \ref{cor.FisnotLLthin}, we conclude that
   	$F$ is $\LL$-large at infinity,
   	which is in contradiction with our assumption on $F$. This ends the proof.
  \end{proof}
  \section{$\Gamma$-cones} \label{sec.suffcondCone}
   {The present section is aimed 
   to demonstrate the criterion
   for $p_\LL$-unboundedness contained in
   Theorem \ref{thm.GammaconeIntro}.}   
   {To this end, as anticipated in the Introduction,
   we need to require
   our PDOs $\LL$ to satisfy another additional assumption}:
   \begin{itemize}
    \item[(D)] there exist two constants 
    $\alpha',\alpha'' > 2$, with $\alpha'<\alpha''$, such that
    \begin{equation} \label{eq.doublingreverseweak} 
	  \alpha'\,\big|\Omega(x,r)\big|
     \leq \big|\Omega(x,2r)\big|
     \leq \alpha''\,\big|\Omega(x,r)\big|
    \end{equation}
    for every $x\in\RN$ and every $r > 0$
	(here and throughout, $|A|$ indicate the standard $N$-di\-men\-sional
	Lebesgue measure in $\RN$ of a Borel set $A\subseteq\RN$).     
  \end{itemize}
  Roughly put, assumption (D) represents a 
   \emph{global} doubling and reverse doubling
   condition for the $N$-volume of $\Gamma$-balls; as we shall see in the next
   Section \ref{sec.homoperator}, such a condition is fulfilled
   when \emph{homogeneous} H\"ormander operators are involved.
   \begin{remark} \label{rem.extendeddoubling}
    It is not difficult to recognize that
    \eqref{eq.doublingreverseweak} in assumption (D) implies
    the following crucial fact: \emph{there exists a constant $\alpha \geq 1$ such that}
    \begin{equation} \label{eq.doublingandreverseBIS}
     \frac{1}{\alpha}\,\bigg(\frac{R}{r}\bigg)^{p}\,\big|\Omega(x,r)\big|
     \leq \big|\Omega(x,R)\big|
     \leq \alpha\,\bigg(\frac{R}{r}\bigg)^{q}\,\big|\Omega(x,r)\big|
    \end{equation}
    \emph{for every $x\in\RN$ and every $0<r<R$, where}
    \begin{equation} \label{eq.pqgreaterthan1}
    p := \log_2(1/\alpha') > 1 \qquad \text{and} \qquad  q = \log_2(\alpha'') > 1.
    \end{equation}	
	As will be clear from the sequel, the r\^ole of assumption
    (D) is only to guarantee the validity of 
    \eqref{eq.doublingandreverseBIS}
    \emph{with $p > 1$}: in fact, in what follows
    we shall only use this relation,
    which could also be taken as an assumption (in place of
    \eqref{eq.doublingreverseweak}). \medskip
    
     Notice that, if
	\eqref{eq.doublingandreverseBIS} holds true
	(for some $\alpha \geq 1$ and $p,\,q > 1$), by taking $R = 2\,r$ one re-obtains
    \eqref{eq.doublingreverseweak} with $\alpha' = \alpha/2^p$ and $\alpha'' = 2^q\,\alpha$;
	 however, if we do not have any information on the value of $\alpha$, we cannot
	 expect that $\alpha' > 2$. Thus, in some sense,
	 the validity of \eqref{eq.doublingandreverseBIS} is a \emph{weaker
	 assumption} if compared to (D).
   \end{remark}
   {With assumption (D) at hand, the proof
   of Theorem
   \ref{thm.GammaconeIntro} will easily follow
   by combining
   Remark \ref{rem.propertiespset} with the next
   non-trivial result.}   
   \begin{theorem} \label{thm.GammaCone}
    Let $K\subseteq\RN$ be a $\Gamma$-cone, according
    to Definition \ref{defi.Gammacone}.
    Then $K$ \emph{is $p_\LL$-unbounded} \emph{(}for the 
    same $p > 1$ in \eqref{eq.pqgreaterthan1}\emph{)}.
   \end{theorem}
   \begin{proof}
    {Let $\mathcal{F} = \{\Omega(z_j,R_j\}_{j\in J}$
    be a family of $\Gamma$-balls contained in
    $K$ satisfying (i) and (ii) in Definition \ref{defi.Gammacone}.   
    Moreover, let
    $\mathbf{c}$ be the constant appearing in the pseudo-triangle inequality
	for $\gamma$ (see Remark \ref{rem.assumptionGdistance}) and let
	$M := 4\mathbf{c}^2 > 4$.}
	
	{Since $\gamma_0(z) = 1/\Gamma_0(z)\to \infty$ as 
	$\|z\|\to\infty$ (see assumption (FS)-(d)), pro\-per\-ties (i) and (ii)
	of $\mathcal{F}$ imply the existence of
	an increasing
	sequence $\{k_j\}_{j\in\N}$ of natural numbers 
	and a real $\delta\in (0,1/M)$ such that} \medskip
	
	(a)\,\,$\gamma_0(z_{k_{j+1}}) > M^2\,\gamma_0(z_{k_j})$ 
	for every $j\in\N$; \medskip
	
	(b)\,\,$R_{k_j}\geq \delta\,\gamma_0(z_{k_j})$ for every $j\in\N$. \medskip
	
	\noindent We then set, for every natural $j$,
	\begin{equation} \label{eq.defiyjrhoj}
	 y_j := z_{k_j}, \qquad \rho_j := \delta\,\gamma_0(z_{k_j}), \qquad
	B_j := \Omega(y_j,\rho_j)
	\end{equation}
	and we consider the set $F_0$ defined as follows:
	$$F_0 := \bigcup_{j = 1}^\infty B_j.$$
	Since, by (b), $R_{k_j}\geq \rho_j$, we derive that
  $B_j \subseteq \Omega(z_{k_j},R_{k_j})\subseteq K$ for every
	$j\in\N$; hence, $F_0\subseteq K$. 
	As a consequence, to prove that $K$ is $p_\LL$-unbounded
	it suffices to show that $F_0$ is $p_\LL$-unbounded
	(for the same $p > 1$ appearing in \eqref{eq.pqgreaterthan1}).
	
	To this end, we choose a sequence
	$\{D_n = \Omega(x_n,r_n)\}_{n}$ of $\Gamma$-balls such that
	$$F_0 \subseteq \bigcup_{n = 1}^\infty D_n$$
	and we prove that, if $p$ is as in \eqref{eq.pqgreaterthan1}, 
	one has (see Definition \ref{def.pset})
	\begin{equation} \label{eq.toproveF0not}
	 \sum_{n = 1}^\infty \big(\Gamma_0(x_n)\,r_n\big)^p =
	 \sum_{n = 1}^\infty \bigg(\frac{r_n}{\gamma_0(x_n)}\bigg)^p = \infty.
	\end{equation}
	Let then $\e \in (0,1/M)$ be fixed and let 
	$\mathcal{A}_\e\subseteq\N$ be defined as follows:
	$$\mathcal{A}_\e := \bigg\{n\in\N:\frac{r_n}{\gamma_0(x_n)} \geq \e\bigg\}.$$
	If $\mathcal{A}_\e$ is infinite, then the claimed \eqref{eq.toproveF0not} is
	obviously true; we thus assume that the set $\mathcal{A}_\e$ is finite
	and we choose a natural $\overline{n} = \overline{n}_\e$ such that
	\begin{equation} \label{eq.relrngammaxn}
	 \frac{r_n}{\gamma_0(x_n)} \leq \e \quad \text{for every $n\geq \overline{n}$}.
	 \end{equation}
	We now prove some technical facts we shall need
	to show that \eqref{eq.toproveF0not} holds. \medskip
	
	\textsc{Claim I:} There exists a natural $\overline{j} = \overline{j}_\e$ such that
	\begin{equation} \label{eq.inclusionBjDn}
	 \bigcup_{j\geq \overline{j}}B_j \subseteq \bigcup_{n\geq \overline{n}}D_n.
	 \end{equation}
	 In fact, let $k\in\N$ be such that
	 $B_k\cap D_n \neq \varnothing$ for some $n \in 
	 \mathcal{J} := \{1,\ldots,\overline{n} - 1\}$
	 and let $z\in B_k\cap D_n$,
	 By the properties of $\gamma$ in Remark \ref{rem.GGammaquasid} we get 
	 \begin{align*}
	 \rho_k & \stackrel{\eqref{eq.defiyjrhoj}}{=} 
	 \delta\,\gamma_0(y_k)
	 \leq \delta\,\mathbf{c}\,\big(\gamma(0,z)+\gamma(z,y_k)\big) \\[0.2cm]
	 & \,\,\leq\,\, \delta\,\mathbf{c}^2\,
	 \big(\gamma(0,x_n)+\gamma(x_n,z)+\gamma(z,y_k)\big) \\[0.2cm]
	& \,\,\leq\,\, \delta\,\mathbf{c^2}\,\big(\rho_k + 
	\max_{n\in \mathcal{J}}(r_n+\gamma_0(x_n))\big)
	\qquad\qquad \big(\text{since $z\in B_k\cap D_n$}\big);
	\end{align*}		
	 as a consequence, since $\delta\,\mathbf{c}^2 < 1/4 < 1$ (
	 by the choice of $\delta$), we obtain
	$$\dsy \rho_k \leq \frac{\max_{n\in \mathcal{J}}(r_n+\gamma_0(x_n))}
	{1-\delta\,\mathbf{c}} =: \tau.$$
	On the other hand, 
    since $\rho_j = \delta\,\gamma_0(y_j)\to\infty$ as 
    $j\to\infty$ (by (b)), it is possible to
	find a natural $\overline{j} = \overline{j}_\e$ such that
	$\rho_j > \tau$ for every $j\geq\overline{j}$; hence
	$$B_j\cap D_n = \varnothing \quad \text{for every $j \geq \overline{j}$
	and every $n < \overline{n}$}.$$
	By taking into account that $\{D_n\}_n$ is a cover
	of $F_0$, we obtain \eqref{eq.inclusionBjDn}. \medskip
	
  \textsc{Claim II:} If $\overline{j}\in\N$ is as in 
  \eqref{eq.inclusionBjDn}, we define
  $$P_j := \{n\geq\overline{n}: B_j\cap D_n \neq \varnothing\}.$$
  Then the following fats hold true: \medskip
  
  $(\star)$\,\,$\dsy \frac{1}{M} \leq \frac{\gamma_0(x_n)}{\gamma_0(y_j)} \leq M$ 
  for every $j\geq\overline{j}$
  and every $k\in P_j$; \medskip
  
  $(\star\star)$\,\,$P_i\cap P_j = \varnothing$ if $i,j\geq 
  \overline{j}$ and $i\neq j$. \bigskip
  
  \noindent As for $(\star)$ we observe that, if 
  $n\in P_j$ (for some $j\geq \overline{j}$)
  and if $z\in B_j\cap D_n\neq\varnothing$, by the properties of $\gamma$
  in Remark \ref{rem.GGammaquasid} (and the choice of $M$ and $\e$) we have
  \begin{align*}
   \gamma_0(x_n) & = \gamma(0,x_n) \leq \mathbf{c}\,
   \big(\gamma(x_n,z)+\gamma(0,z)\big) \\[0.2cm]
  &  \leq \mathbf{c}\,\big(\gamma(x_n,z) + \mathbf{c}\,
  \big(\gamma(0,y_j) + \gamma(y_j,z)\big)\big) \\[0.2cm]
   & \leq \mathbf{c}^2\,\big(r_n + \rho_j + \gamma_0(y_j)\big) 
   \qquad\qquad\quad
    \big(\text{since $z\in B_j\cap D_n$}\big) \\[0.2cm]
   & \leq \mathbf{c}^2\,\big((1+\delta)\,\gamma_0(y_j)+\e\,
   \gamma_0(x_n)\big) \qquad
		\big(\text{see \eqref{eq.defiyjrhoj} and \eqref{eq.relrngammaxn}}\big) \\
   & \leq 2\mathbf{c}^2\,\gamma_0(y_j) + \frac{1}{2}\,\gamma_0(x_n).
  \end{align*}
  From this, we derive that
  $$\frac{\gamma_0(x_n)}{\gamma_0(y_j)} \leq 4\mathbf{c}^2 = M,$$
  which is precisely the second inequality in 
  $(\star)$; by arguing analogously, one can
  prove the first one too. We now turn to prove $(\star\star)$. 
  
  To this end, we argue by contradiction and we assume the existence
  of two indexes $i,j\geq \overline{j}$, with (to fix ideas) $j > i$, such that
  $P_i\cap P_j \neq \varnothing$. If $n \geq \overline{n}$ is a natural
  number belonging to such an intersection,
  from $(\star)$ we derive that
  $$\frac{\gamma_0(y_j)}{\gamma_0(y_i)} = 
  \frac{\gamma_0(y_j)}{\gamma_0(x_n)}\cdot\frac{\gamma_0(x_n)}{\gamma_0(y_i)}
  \leq M^2,$$
  which is in contradiction with (a). \medskip
  
  \textsc{Claim III:} There exists a real constant $\zeta > 0$ such that
  \begin{equation} \label{eq.mainestimdoubling}
   \bigg(\frac{\gamma_0(y_j)}{r_n}\bigg)^p\cdot\frac{|D_n|}{|B_j|}
   \leq \zeta
   \quad \text{for any $j\geq\overline{j}$ and $n\in P_j$}.
  \end{equation}
  In fact, by using $(\star)$ and the pseudo-triangle inequality
  for $\gamma$, it is possible to find a real $\zeta_1 > 1$, only depending
  on $M$ and $\mathbf{c}$, such that
  \begin{equation} \label{eq.touseinclusionGammaballs}
   \Omega(x_n,\gamma_0(y_j))\subseteq \Omega(y_j,\zeta_1\gamma_0(y_j))
	\quad \text{for any $j\geq\overline{j}$ and $n\in P_j$}.
	\end{equation}
	From this, by applying the second inequality in
	\eqref{eq.doublingandreverseBIS}, we obtain
	\begin{align*}
	 \frac{|D_n|}{|B_j|}
	 & = \frac{\big|\Omega(x_n,r_n)\big|}{\big|\Omega(y_j,\rho_j)\big|}
	 \stackrel{\eqref{eq.defiyjrhoj}}{=} \frac{\big|\Omega(x_n,r_n)\big|}
	 {\big|\Omega(y_j,(\delta/\zeta_1)\cdot(\zeta_1\gamma_0(y_j))\big|} \\[0.2cm]
	 & \leq \alpha\,\bigg(\frac{\zeta_1}{\delta}\bigg)^q
	 \cdot\frac{\big|\Omega(x_n,r_n)\big|}
	 {\big|\Omega(y_j,\zeta_1\gamma_0(y_j))\big|} 
	 \qquad \big(\text{since $\delta/\zeta_1 < 1$}	\big)\\[0.2cm]
	 &  \leq
	 \alpha\,\bigg(\frac{\zeta_1}{\delta}\bigg)^q
	 \cdot\frac{\big|\Omega(x_n,r_n)\big|}
	 {\big|\Omega(x_n,\gamma_0(y_j))\big|} := \big(\bigstar\big)
	  \qquad \big(\text{by \eqref{eq.touseinclusionGammaballs}}\big)
	\end{align*}
	On the other hand, since $n\geq\overline{n}$, again by $(\star)$ we have
	$$\frac{r_n}{\gamma_0(y_j)} \leq \e\cdot\frac{\gamma_0(x_n)}
	{\gamma_0(y_j)} \leq \e\,M < 1;$$
	we are then entitled to use the reverse doubling
	condition \eqref{eq.doublingandreverseBIS}, which gives
	$$\big(\bigstar\big) \leq \zeta\,\bigg(\frac{r_n}{\gamma_0(y_j)}\bigg)^p$$
	for some universal constant $\zeta$ not depending on $j$ and $n$. \bigskip
	
	Now we have established all these claims, we can easily achieve the proof of
	the needed \eqref{eq.toproveF0not}.
	Indeed, by $(\star)$ and $(\star\star)$ in Claim II, we have
	\begin{align*}
	 \sum_{n\geq \overline{n}}\bigg(\frac{r_n}{\gamma_0(x_n)}\bigg)^p & 
	 \geq \sum_{j\geq \overline{j}}\sum_{n\in P_j}\bigg(\frac{r_n}
	 {\gamma_0(x_n)}\bigg)^p 
	 \geq \frac{1}{M^p}\,\sum_{j\geq \overline{j}}\sum_{n\in P_j}
	 \bigg(\frac{r_n}{\gamma_0(y_j)}\bigg)^p
	 =: \big(\bigstar\big)
	\end{align*}
	On the other hand, if $j\geq\overline{j}$ is fixed, the family
	$\{D_n\}_{n\in P_j}$ is a cover of the set $B_j$
	(see \eqref{eq.inclusionBjDn}); as a consequence, by exploiting estimate
	\eqref{eq.mainestimdoubling}, we finally obtain
	\begin{align*}
	 \big(\bigstar\big) & \geq \big(\zeta\,M^p\big)^{-1}\,\sum_{j\geq\overline{j}}
	 \frac{1}{|B_j|}\,
	 \sum_{n\in P_j}|D_n| \\[0.2cm]
	 & \geq \big(\zeta\,M^p\big)^{-1}\,\sum_{j\geq\overline{j}} 1 = \infty.
	\end{align*}
	This is precisely the desired \eqref{eq.toproveF0not}, and the proof is complete.
  \end{proof}
  With Theorem \ref{thm.GammaCone} at hand, we can now
  prove Theorem \ref{thm.GammaconeIntro}.
  \begin{proof} [Proof (of Theorem \ref{thm.GammaconeIntro}).]
   Let $F,\,K$ be as in the statement of the theorem. Since,
   by assumption, $K$ is a $\Gamma$-cone, we infer from
   Theorem \ref{thm.GammaCone} that $K$ is $p_\LL$-unbounded
   (for the same $p$ as in \eqref{eq.pqgreaterthan1}); on the other hand,
   as $K\subseteq F$, Remark \ref{rem.propertiespset}-(2.)
   shows that also $F$ is $p_\LL$-unbounded, and
   the proof is complete.
  \end{proof}
  \section{The case of homogeneous H\"ormander operators}  \label{sec.homoperator}
  The aim of this final section is to show that
  any H\"ormander's operator sum of squares of \emph{homogeneous
  vector fields} satisfies all the assumptions
  (H1)-to-(H3), (FS), (G), (L) and (D) introduced in the previous sections. \medskip
  
  To this end, we fix once and for all a family
  $\mathcal{X} = \{X_1,\ldots,X_m\}$ of linearly independent
  smooth vector fields on $\RN$, with $N\geq 3$, satisfying
  assumptions (I) and (II) in the Introduction. Moreover, we let
  $$Q := \sum_{j = 1}^N\sigma_j \geq 3,$$
  be the homogeneous dimension of $\RN$ with respect to the family of dilations
  $$\delta_\lambda: \RN\to\RN, \quad \delta_\lambda(x) = (\lambda^{\sigma_1}x_1,
  \ldots,\lambda^{\sigma_N}x_N).$$
  We then denote by $\LL$ the operator naturally associated 
  with $\mathcal{X}$, that is,
  $$\LL = \sum_{j = 1}^m X_j^2.$$
 \paragraph{Assumptions (H1)-to-(H3).} 
  It is easy to recognize that $\LL$ satisfies all the structural assumptions
 (H1)-to-(H3) introduced in Section \ref{sec:mainassnot}: indeed, a direct 
 computation shows that
 $\LL$ takes the form \eqref{eq:PDOformintro}, with $V \equiv 1$ and
 $$A(x) = S(x)\cdot S(x)^T, 
 \quad \text{where $S(x) = \big(X_1(x)\cdots X_m(x)\big)$};$$
 as a consequence, $\LL$ is degenerate-elliptic. Moreover, 
 the validity of H\"ormander's Rank Condition
 easily implies that
 $\LL$ is non-totally degenerate and that $\LL$
 is $C^\infty$-hypoelliptic on every open subset of $\RN$
 (by H\"ormander's Theorem). 
 
 \paragraph{Assumption (FS).}
  We now prove that $\LL$ also satisfies assumption (FS). 
  First of all, by exploiting
  \cite[Theorem 1.1]{BiagiBonfLast}, we get the existence
  of a function $\Gamma(x;y)$, defined
  out of the diagonal of $\RN\times\RN$, such that
  \begin{itemize}
   \item $\Gamma$ is smooth and strictly positive on its domain of definition;
   \item $\Gamma(x;y) = \Gamma(y;x)$ for every $x,y\in\RN$ with $x\neq y$;
   \item for every fixed $x\in\RN$, $y\mapsto \Gamma(x;y) = 
   \Gamma_x(y)\in L^1_{\loc}(\RN)$ and
   $$\int_{\RN}\Gamma(x;y)\,\LL\varphi(y)\,\d y = -\varphi(x) \quad
   \text{for every $\varphi\in C_0^\infty(\RN,\R)$};$$
   \item $\Gamma(x;\cdot)$ vanishes at infinity (uniformly for $x$ in compact sets);
   \item $\Gamma$ has the (joint) homogeneity property
   \begin{equation} \label{eq.crucialGammahom}
    \Gamma(\dela(x);\dela(y)) = \lambda^{2-Q}\Gamma(x;y) \quad 
    \text{for all $x,y\in\RN$}.
   \end{equation}
  \end{itemize}
  Furthermore, by using the results in the very recent paper
  \cite{BiBoBra} (see, precisely, Theorem 1.3-(V)), we know that
  $\Gamma(x;\cdot)$ has a pole at $x$, i.e.,
  $$\lim_{y\to x}\Gamma(x;y) = \infty \qquad\text{for any fixed $x\in\RN$}.$$
  Summing up, $\LL$ satisfies assumption (FS).
  
  \paragraph{Assumption (G).} In this paragraph 
  prove that
	$\LL$ also satisfies assumption (G). To this end,
  we first need to remind some results
  concerning the so-called
  control distance associated with the family $\mathcal{X} = \{X_1,\ldots,X_m\}$. \medskip
  
  Let $f:[0,T]\to\RN$ be a Lipschitz curve. We say that $f$ is $\mathcal{X}$-subunit if
  $$\langle \dot f(t),\xi\rangle \leq \sum_{j = 1}^m \langle X_j(f(t)), \xi\rangle^2,
  \quad \text{for a.e.\,$t\in [0,T]$ and every $\xi\in\RN$}.$$
  Denoting by $\mathcal{S}(\mathcal{X})$ the set of all $X$-subunit curves, we can define
  $$d_\mathcal{X}(x,y) := \inf\bigg\{T > 0\,:\,\text{$\exists\,\,
  f\in\mathcal{S}(\mathcal{X)}$ such that $f(0) = x$ and $f(T) = y$}\bigg\}.$$
  Since $X_1,\ldots,X_m$ satisfy H\"ormander's Rank Condition, the function
   $d_{\mathcal{X}}$
  is finite for every $x,y\in\RN$ and it defines a \emph{distance} on $\RN$, which is usually
  referred to as the \emph{control distance} associated with $\mathcal{X}$
  (see, e.g., \cite[Chapter 19]{BLUlibro} and the references therein).
  Moreover, since the $X_j$s are $\dela$-homogeneous of degree $1$,
  \begin{equation} \label{eq.dXishomog}
   d_\mathcal{X}(\dela(x),\dela(y)) = \lambda\,d_\mathcal{X}(x,y).
  \end{equation}
  For every fixed $x\in\RN$ and every $r > 0$, we indicate by
  $B_{\mathcal{X}}(x,r)$ the (open) 
  $d_\mathcal{X}$-ball with centre $x$ and radius $r$, that is,
  $$B_{\mathcal{X}}(x,r) := \{y\in\RN: d_\mathcal{X}(x,y) < r\}.$$
  By \eqref{eq.dXishomog}, it is easy to see that $d_\mathcal{X}$-balls are preserved
  by dilations, that is,
  \begin{equation} \label{eq.dxballshomogeneous}
   \dela\big(B_{\mathcal{X}}(x,r)\big) = B_{\mathcal{X}}(\dela(x),\lambda r);
   \end{equation}
  from this, by using a deep result by
  Nagel, Stein and Wainger \cite{NagelSteinWainger}, one obtains
  the following \emph{global} estimates for the $N$-volume of $d_\mathcal{X}$-balls
  (see also \cite[Thm.\,B]{BiBoBra}).	
  \begin{theorem} \label{main.ThmNSW}
   There exist a real constant $c_1 \geq 1$ such that
   \begin{equation} \label{eq.estimlambdaBx}
    \frac{1}{c_1}\,\sum_{j = N}^Q F_j(x)\,r^j 
    \leq \big|B_{\mathcal{X}}(x,r)\big| \leq c_1\,\sum_{j = N}^Q F_j(x)\,r^j
   \end{equation}
   for every $x\in\RN$ and every $r > 0$. Here, 
   the functions $F_j$ are positive continuous functions
   and, for every $j$, $F_j$ is $\dela$-homogeneous of degree $Q-j$.
  \end{theorem}
  \begin{proof}
	First of all, we need to introduce some notations borrowed from
	\cite{NagelSteinWainger}
	(see also \cite[Section 4.2]{BramantiBOOK}):
	if $p \in \N$ and 
	$I = (i_1,\ldots,i_p)$ is a multi-index of length $p$ (i.e., $I$ is a vector
	in $\R^p$ with non-negative integer components), we define
	$$X_{I} := \big[X_{i_1}\cdots[X_{i_{p-1}},X_{i_p}]\cdots\big]
	\quad \text{and} \quad |I| := p.$$
	Furthermore, if $B = (I_1,\ldots,I_N)$ is a $N$-tuple
	of multi-indexes, we set
	$$\lambda_B(x) := \det\Big(X_{I_1}(x)\cdots X_{I_N}(x)\Big) \quad
	\text{and} \quad l(B) := \sum_{j = 1}^N|I_j|\geq N.$$
	Finally, if $s$ is any natural number, we denote by $\mathcal{B}_s$ the set of 
	all the possible $N$-tuples $B = (I_1,\ldots,I_N)$ of multi-indexes
	with $|I_j|\leq s$ for every $j = 1,\ldots,N$.
	
	We now observe that, by assumptions (H1) and (H2), the Lie algebra
	generated by $X_1,\ldots,X_m$ is nilpotent of step $s := \sigma_N$; as a consequence,
	if $\mathcal{U}\subseteq\RN$ is any (fixed) bounded and 
	connected open neighborhood of $0$,
	\cite[Theorem 1]{NagelSteinWainger} provides us with a small $r_0 > 0$
	and a real constant $c_1 \geq 1$ such that
	\begin{equation} \label{eq.localestimNSWballs}
	 \frac{1}{c_1}\,\sum_{B\in\mathcal{B}_s}
	\big|\lambda_B(x)\big|\,r^{l(B)}
	 \leq \big|B_{\mathcal{X}}(x,r)\big|
	 \leq c_1\,\sum_{B\in\mathcal{B}_s}
	 \big|\lambda_B(x)\big|\,r^{l(B)}
	\end{equation}
	for every $x\in\mathcal{U}$ and every $r > 0$ such that $r \leq r_0$. 
	
	We claim that,
	as a consequence of the homogeneity of $X_1,\ldots,X_m$,
	estimate \eqref{eq.localestimNSWballs} actually holds for
	every $x\in\RN$ and every $r > 0$. 
	Indeed, 
	if $x\in\RN$ is arbitrary fixed and if $r > 0$, it is possible to choose
	$\lambda = \lambda_{x,r} > 0$ such that 
	$$\text{$x' = \dela(x)\in\mathcal{U}$
	and $r' = \lambda\,r \leq r_0$;}$$
   thus, \eqref{eq.localestimNSWballs}
   holds with $x'$ and $r'$ in place of
   $x$ and $r$. Now, by \eqref{eq.dxballshomogeneous} we have
   \begin{equation} \label{eq.hom_measuredBall}
    \big|B_\mathcal{X}(x',\lambda\,r)\big| = 
    \big|\dela(B_\mathcal{X}(x,r))\big| = \lambda^Q\,\big|B_\mathcal{X}(x,r)\big|;
   \end{equation}
   on the other hand, if $B = (I_1,\ldots,I_N)\in\mathcal{B}_s$,
   the homogeneity of $X_1,\ldots,X_m$ with respect to $\dela$
   implies that (see, e.g., \cite[Corollary 1.3.6]{BLUlibro})
   \begin{equation} \label{eq.hom_lambdaIrlB}
   \begin{split}
    & \big|\lambda_B(x')\big|\,(r')^{l(B)}  = 
    \big|\det\big(X_{I_1}(x')\cdots X_{I_N}(x')\big)\big|\,(r')^{l(B)} \\[0.1cm]
    & \qquad = \det\Big(\lambda^{-|I_1|}\,\dela\big(X_{I_1}(x)\big)\cdots 
    \lambda^{-|I_N|}\,\dela\big(X_{I_N}(x)\big)\Big)\,(\lambda\,r)^{l(B)} \\[0.1cm]
    & \qquad = \lambda^{Q-l(B)}\,\lambda_B(x)\,(\lambda\,r)^{l(B)} 
    = \,\lambda^{Q}\,\lambda_B(x)\,r.
   \end{split}
   \end{equation}
   By combining \eqref{eq.hom_measuredBall} with 
   \eqref{eq.hom_lambdaIrlB}, we conclude that the validity
   of \eqref{eq.localestimNSWballs} for $x'$ and $r'$ implies
   the validity of the same estimate for $x$ and $r$, as claimed.
   
   To complete the demonstration of the theorem we observe that, if
   $B\in\mathcal{B}_s$, the computation carried out in
   \eqref{eq.hom_lambdaIrlB} shows that
   $|\lambda_B|$ is a continuous $\dela$-homogeneous function of degree $Q-l(B)$;
   as a consequence, we have
   $$|\lambda_B| \equiv 0 \quad \text{for every $B\in\mathcal{B}_s$ with $l(B) > Q$}.$$
	Thanks to this last fact, we can write (for $x\in\RN$ and $r > 0$)
	\begin{align*}
	 \sum_{B\in\mathcal{B}_s}\big|\lambda_B(x)|\,r^l(B)
	 & = \sum_{j = N}^Q\underbrace{\bigg(\sum_{\begin{subarray}{c}
	 B\in\mathcal{B}_s \\
	 l(B) = j
	 \end{subarray}}\big|\lambda_B(x)\big|\bigg)}_{=:F_j(x)}r^{j}
	 \equiv \sum_{j = N}^Q F_j(x)\,r^j.
	\end{align*}
	Note that, by definition, any $F_j$ is $\dela$-homogeneous
	of degree $Q-j$ (since
	$|\lambda_B|$ is $\dela$-homogeneous of degree $Q-l(B) = Q-j$
	if $l(B) = j$). This ends the proof.
  \end{proof}
  We now turn to show how $d_\mathcal{X}$ (and the associated balls)
  are related with the fundamental solution $\Gamma$.    
  To this end, we introduce the following functions:	
    \begin{equation} \label{eq.defLambdaE}
   \Lambda(x,r) := \sum_{j = N}^Q F_j(x)\,r^j, \qquad
   E(x,r) := \frac{\Lambda(x,r)}{r^2} = \sum_{j = N}^Q F_j(x)\,r^{j-2}.
  \end{equation}
  \begin{remark} \label{rem.propertiesELambda}
   We list, for future reference, some useful properties 
   of $\Lambda$ and $E$:
   \begin{itemize}
    \item[(a)] For every fixed $x$, both $\Lambda(x,\cdot)$ and
    $E(x,\cdot)$ are strictly increasing on $(0,\infty)$;
    \item[(b)] For every $x\in\RN$ and every $0 < r < R$ we have
    \begin{flalign}
     & \mathrm{(b)_1}\quad\bigg(\frac{R}{r}\bigg)^N\,\Lambda(x,r) \leq \Lambda(x,R) \leq
    \bigg(\frac{R}{r}\bigg)^Q\,\Lambda(x,r) && \label{eq.doublingLambda} \\[0.2cm]
     & \mathrm{(b)_2}\quad\bigg(\frac{R}{r}\bigg)^{N-2}\,E(x,r) \leq E(x,R) \leq
    \bigg(\frac{R}{r}\bigg)^{Q-2}\,E(x,r) \label{eq.doublingE}
    \end{flalign}
    \item[(c)] If $B_{\mathcal{X}}(x,r)\subseteq B_\mathcal{X}(y,\rho)$ for some
    $x,y\in\RN$ and $r,\rho\in(0,\infty)$, we have
    \begin{align} \label{eq.ELambdainclusion}
     \Lambda(x,r) \leq c_1^2\,\Lambda(y,\rho) \quad
     \text{and} \quad 
     E(x,r) \leq c_1^2\,\bigg(\frac{\rho}{r}\bigg)^2\,E(y,\rho).
    \end{align}
    \item[(d)] There exists
	a strictly positive constant $\omega_Q > 0$ such that
	\begin{flalign}
     & \mathrm{(d)_1}\quad\Lambda(0,r) = \omega_Q\,r^Q \qquad\text{for every
     $r > 0$}; \label{eq.Lambda0rrQ} \\[0.2cm]
     & \mathrm{(d)_2}\quad\Lambda(x,r)\geq\omega_Q\,r^Q \qquad\text{for
     every $x\in\RN$ and every $r > 0$}. \label{eq.LambdaxrrQ}
    \end{flalign}
    In fact, since any function $F_j$ appearing
    in \eqref{eq.defLambdaE} is non-negative, continuous
    and $\dela$-homogeneous of degree $Q-j$ (see
    Theorem \ref{main.ThmNSW}), we have \medskip
    
    $(\ast)$\,\,$F_j(0) = 0$ for every $j = N,\ldots,Q-1$; \vspace*{0.02cm}
    
    $(\ast\ast)$\,\,$F_Q(x) = F_Q(0) = \omega_Q\geq 0$ for every $x\in\RN$. \medskip
    
    As a consequence, by the very definition of $\Lambda$
    (see \eqref{eq.defLambdaE}), for every
    $x\in\RN$ and for every $r > 0$ we can write
    $$\Lambda(0,r) = \omega_Q\,r^Q \quad
    \text{and} \quad \Lambda(x,r) \geq \omega_Q\,r^Q,$$
	with $\omega_Q\geq 0$. From this, since \eqref{eq.estimlambdaBx}
	implies that
	$$\omega_Q = \Lambda(0,1) \geq \frac{1}{c_1}\,\big|B_\mathcal{X}(0,1)\big|
	> 0,$$
	we obtain both \eqref{eq.Lambda0rrQ} and \eqref{eq.LambdaxrrQ}.
   \end{itemize}
  \end{remark}
  By combining \cite[Theorem 1.3-(III)]{BiBoBra} with the above
  Theorem \ref{main.ThmNSW}, we
  are able to demonstrate the following key result.
  \begin{theorem} \label{thm.globalSanchez}
	 There exists a real constant $c_2 > 0$ such that
	 \begin{equation} \label{eq.globalSanchez}
	  \frac{1}{c_2}\,\frac{d_\mathcal{X}^2(x,y)}{\Lambda(x,d_\mathcal{X}(x,y))} \leq
	 \Gamma(x;y) \leq c_2\,\frac{d_\mathcal{X}^2(x,y)}{\Lambda(x,d_\mathcal{X}(x,y))} \quad
	 \text{for all $x\neq y$}.
	 \end{equation}
	\end{theorem}
	\begin{proof}
	First of all, since
	we are assuming that
	the operator $\LL$ is defined on some
    space $\RN$ with $N\geq 3$,
	we are entitled to apply \cite[Theorem 1.3-(III)]{BiBoBra}:
	as a consequence, for every $x, y\in\RN$ with $x\neq y$ we have
	\begin{equation} \label{eq.estimateBiBoBra}
	 C^{-1}\frac{d_{\XX}(x,y)^{2}}{\big|B_{\XX}\big(x,d_{\XX}(x,y)\big)\big|}\leq
   \Gamma (x;y)\leq C\,\frac{d_{\XX}(x,y)^{2}}{\big\vert B_{\XX}
   (x,d_{X}(x,y))\big\vert},
	\end{equation}
	where $C \geq 1$ is a suitable structural constant. 
	By combining \eqref{eq.estimateBiBoBra} with the global estimate
   \eqref{eq.estimlambdaBx} for $|B_{\XX}(x,r)|$
	(holding true for any $r > 0$),
	 we immediately
	obtain the desired \eqref{eq.globalSanchez}
	(with $c_2 := C\cdot c_1$). This ends the proof.
   \end{proof}
   With Theorem \ref{thm.globalSanchez} at hand,
   we can now prove that $\LL$
   fulfills assumption (G).
   \begin{proposition} \label{prop.assumptionGhom}
	 $\LL$ satisfies assumption \emph{(G)} introduced in Section \ref{sec:LLthinpset}.
	\end{proposition}
	\begin{proof}
	 According to Remark \ref{rem.assumptionGdistance}, $\LL$ fulfill
	 assumption (G) \emph{if and only if} the reciprocal 
	 function $\gamma(x,y) = 1/\Gamma(x;y)$
	 (with the convention $\gamma(x,x) = 0$) satisfies a pseudo-triangle inequality;
	 on the other hand, by Theorem \ref{thm.globalSanchez}, we have
	 $$\frac{1}{c_1}E(x,d_\mathcal{X}(x,y)) \leq \gamma(x,y) \leq c_2\,E(x,d_\mathcal{X}(x,y))
	 \quad \text{for every $x,y\in\RN$}.$$
	 Thus, to prove the proposition, it suffices 
	 to show that there exists $\mathbf{c} > 1$ such that,
	 for every $x,y,z\in\RN$, the following inequality holds true:
	 \begin{equation} \label{eq.toprovepseudoE}
	  E(x,d_\mathcal{X}(x,y)) \leq \mathbf{c}\big(E(x,d_\mathcal{X}(x,z))
	 + E(z,d_\mathcal{X}(z,y))\big).
	 \end{equation}
	 First of all we observe that, 
	 since $d_\mathcal{X}$ satisfies a genuine triangle inequality, for every
	 $x,y,z\in\RN$ we have
	 (see Remark \ref{rem.propertiesELambda}-(a))
	 $$E(x,d_\mathcal{X}(x,y)) \leq E\big(x, d_\mathcal{X}(x,z) + 
	 d_\mathcal{X}(z,y)\big) =: \big(\bigstar\big)$$
	 as a consequence, 
	 if $d_\mathcal{X}(z,y)\leq d_\mathcal{X}(x,z)$, we obtain
	 (see also Remark \ref{rem.propertiesELambda}-(b))
	 \begin{align*}
	  \big(\bigstar\big) &\,\,\leq\,\,E(x,2d_\mathcal{X}(x,z))
	  \stackrel{\eqref{eq.doublingE}}{\leq} 2^{Q-2}\,E(x,d_\mathcal{X}(x,z)) \\[0.2cm]
	  & \,\,\leq\,\,2^{Q-2}\,\big(E(x,d_\mathcal{X}(x,z))
	 + E(z,d_\mathcal{X}(z,y))\big).
	 \end{align*}
	 If, instead, $d_\mathcal{X}(z,y) > d_\mathcal{X}(x,z)$, from the obvious
	 fact that 
		  $B_\mathcal{X}(x,2d_\mathcal{X}(z,y))$ is 
		  included in $B_\mathcal{X}(z,3d_\mathcal{X}(z,y))$
		  we get (see also Remark 
	 \ref{rem.propertiesELambda}-(c))
	 \begin{align*}
	  \big(\bigstar\big) & \,\,\,\,\leq \,\,\,\,
	   E(x,2d_\mathcal{X}(z,y)) 
	  \stackrel{\eqref{eq.ELambdainclusion}}{\leq} 
	  c_1^2\,\bigg(\frac{3}{2}\bigg)^2\,E(z,3\,d_\mathcal{X}(z,y)) \\[0.2cm]
	  & \stackrel{\eqref{eq.doublingE}}{\leq} 
	  c_1^2\,\bigg(\frac{3}{2}\bigg)^2\,3^{Q-2}\,E(z,d_\mathcal{X}(z,y)) \\[0.2cm]
	  & \,\,\,\,\leq\,\,\,\, 
	  \bigg(\frac{3^{Q}c_1^2}{4}\bigg)\,\big(	E(x,d_\mathcal{X}(x,z))
	 + E(z,d_\mathcal{X}(z,y))\big).
	 \end{align*}
	Setting $\mathbf{c} := \max\{2^{Q-2}, 3^{Q}\,c_1^2/4\}$, 
	we obtain the desired \eqref{eq.toprovepseudoE}.
	\end{proof}
	\paragraph{Assumption (L).} In this paragraph we prove that $\LL$
	satisfies the Liouville-type theorem in assumption (L): a bounded
	$\LL$-harmonic function on $\RN$ is constant.
	\begin{proposition} \label{prop.LiouvHom}
	 Let $u\in\mathcal{L}(\RN)$ be a $\LL$-harmonic function
	 on $\RN$. If $u$ is boun\-ded (above or below),
	 then $u$ is constant throughout $\RN$.
	\end{proposition}
	{One demonstration of Proposition \ref{prop.LiouvHom} can be found
	in \cite{LancoK}; however, we present below another prove
	of this result, which is almost self-contained.}
	\begin{proof}
	 By \cite[Theorem 3.2]{BiagiBonfLast}, it is possible to find a homogeneous Carnot
	 group $\G = (\R^H,*,d_\lambda)$ on $\R^H$ (for a suitable $H > N$) and a system
	 $\mathcal{Z} = \{Z_1,\ldots,Z_m\}$ of Lie-generator for $\LieG$ such that,
	 setting $\Delta_\G = \sum_{j = 1}^mZ_j^2$, one has
	 $$\Delta_\G(f \circ \pi) = (\LL f)\circ\pi \qquad 
	 \text{for every $f\in C^\infty(\RN,\R)$}$$
	 (here, $\pi:\R^H\to\RN$ is the canonical projection of $\R^H$ onto the first
	 $N$ variables).
	Thus, since $u\in\mathcal{H}_\LL(\RN)$, the function
	$v := u\circ\pi$ is $\Delta_\G$-harmonic on $\G\equiv\R^H$.
	
	On the other hand, since (by assumption) 
	$u$ is bounded (from above or from below), then the same is true of $v$; as a consequence,
	by the classical Liouville Theorem on Carnot groups 
	(see, e.g., \cite[Theorem 5.8.2]{BLUlibro}), we conclude
	that $v$ is constant throughout $\R^H$, whence $u$ is constant on $\RN$.
	This ends the proof.
	\end{proof}
	\paragraph{Assumption (D).} In this last paragraph of the section we prove
	that $\LL$ fulfills assumption (D). Actually, according to Remark
	\ref{rem.extendeddoubling}, we directly show that
	the super-level sets of $\Gamma$ satisfy the doubling/reverse doubling conditions in
	 \eqref{eq.doublingandreverseBIS}.
	
	To this end we first observe that, since the function $E(x,\cdot)$ is strictly increasing
	on $(0,\infty)$ for every fixed $x\in\RN$ (see Remark \ref{rem.propertiesELambda}),
	we can define
	$$H(x,\cdot) := \big(E(x,\cdot)\big)^{-1} \quad \text{on $(0,\infty)$}.$$
	Obviously, $H$ is strictly increasing on $(0,\infty)$;
	moreover, it satisfies the ``dual'' property
	of \eqref{eq.doublingE}, that is, for every $x\in\RN$ and every $0 < r < R$ we have
	\begin{equation} \label{eq.dualdoublingH}
	 \bigg(\frac{R}{r}\bigg)^{\frac{1}{Q-2}}H(x,r)
	 \leq H(x,R) \leq \bigg(\frac{R}{r}\bigg)^{\frac{1}{N-2}}H(x,r).
	\end{equation}
	By means of such a function (and of Theorem \ref{thm.globalSanchez}),
	we can write a precise relation between $\Gamma$-balls and $d_\mathcal{X}$-balls: in fact,
	since $\gamma(x,y) = 1/\Gamma(x,y)$ 
	can be estimated (from above and from below) by $E(x,d_\mathcal{X}(x,y))$,
	we have (see Remark \ref{rem.GGammaquasid})
	\begin{equation} \label{eq.relationGammaballdball}
	B_{\mathcal{X}}\big(x,H(x,r/c_2)\big) 
	\subseteq \Omega(x,r) \subseteq B_{\mathcal{X}}\big(x, H(x,c_2r)\big)
	\end{equation}		
	for every $x\in\RN$ and every $r > 0$ (here, $c_2$ is the constant
	in Theorem \ref{thm.globalSanchez}). As a consequence of this identity, we easily
	obtain the following lemma.
	\begin{lemma} \label{lem.measureGammaball}
	 There exists an absolute constant $c_3 \geq 1$ such that
	 \begin{equation} \label{eq.measureGammaball}
	  \frac{1}{c_3}\,\big(r\,H^2(x,r)\big) \leq
	  \big|\Omega(x,r)\big| \leq c_3\,\big(r\,H^2(x,r)\big)
	 \end{equation}
	 for every $x\in\RN$ and every $r > 0$.
	\end{lemma}
	\begin{proof}
	 Let $x\in\RN$ be fixed and let $r > 0$. 
	 By the above
	\eqref{eq.relationGammaballdball} and Theorem \ref{main.ThmNSW},
	we have (see also \eqref{eq.defLambdaE} and remind that
	$H(x,\cdot)$ is the inverse of $E(x,\cdot)$)
	\begin{align*}
	 \big|\Omega(x,r)\big| & \,\,\,\geq\,\,\, \big|
	 B_{\mathcal{X}}\big(x,H(x,r/c_2)\big)\big|
	 \stackrel{\eqref{main.ThmNSW}}{\geq} \frac{1}{c_1}\,\Lambda\big(x,H(x,r/c_2)\big) \\[0.2cm]
	 & \stackrel{\eqref{eq.defLambdaE}}{=} \frac{1}{c_1}\,E\big(x, H(x, r/c_2)\big)\cdot H^2(x,r/c_2) \\[0.2cm]
	 & \,\,\,=\,\,\,\frac{1}{c_1\,c_2}\,\big(r\,H^2(x,r/c_2)\big) = \big(\bigstar\big)
	\end{align*}
	From this, by the second inequality in
	 \eqref{eq.dualdoublingH}, we obtain (remind that $c_2\geq 1$)
	\begin{align*}
	 \big(\bigstar\big) \geq \frac{1}{c_1\,c_2}\,c_2^{2/(N-2)}\,\big(r\,H^2(x,r)\big)
	 = \frac{1}{c_3}\,\big(r\,H^2(x,r)\big).
	\end{align*}
	The second inequality in \eqref{eq.measureGammaball}
	 can be demonstrated
	 analogously.
	\end{proof}
	We can now prove that $\LL$ satisfies \eqref{eq.doublingandreverseBIS}.
	\begin{proposition} \label{prop.LLhomsatisfiesD}
	  There exists an absolute constant $c_4 \geq 1$ such that
	  \begin{equation} \label{eq.LLhomsatisfiesD}
	   \frac{1}{c_4}\,
		\bigg(\frac{R}{r}\bigg)^{\frac{Q}{Q-2}}\,\big|\Omega(x,r)\big|
	   \leq \big|\Omega(x,R)\big| 
	   \leq c_4\,\bigg(\frac{R}{r}\bigg)^{\frac{N}{N-2}}\,\big|\Omega(x,r)\big|
	  \end{equation}
	  for every $x\in\RN$ and every $R,r\in (0,\infty)$ with $r < R$.
	\end{proposition}
	\begin{proof}
	 Let $x\in\RN$ be fixed and let $r,R\in(0,\infty)$ be such that
	 $r < R$. By com\-bi\-ning Lemma \ref{lem.measureGammaball}
	 with the first inequality in \eqref{eq.dualdoublingH}, we obtain
	 \begin{align*}
	  \big|\Omega(x,R)\big| & \,\,\,\,\geq\,\,\, \frac{1}{c_3}\,\big(R\,H^2(x,R)\big)
	  \stackrel{\eqref{eq.dualdoublingH}}{\geq} 
	  \frac{1}{c_3}\,\bigg(\frac{R}{r}\bigg)^{2/(Q-2)}\,\big(R\,H^2(x,r)\big) \\[0.2cm]
	  & \,\,\,\,=\,\,\, \frac{1}{c_3}\,\bigg(\frac{R}{r}\bigg)^{Q/(Q-2)}\,\big(r\,H^2(x,r)\big) \\[0.2cm]
	  & \stackrel{\eqref{eq.measureGammaball}}{\geq} 
	  \frac{1}{c_3^2}\,\bigg(\frac{R}{r}\bigg)^{Q/(Q-2)}\,\big|\Omega(x,r)\big|.
	 \end{align*}
	 The second inequality in \eqref{eq.LLhomsatisfiesD} can be proved
	 analogously.
	\end{proof}
	Gathering together all the facts proved in these paragraphs
	we obtain the following result, which is a restatement
	of Theorem \ref{thm.mainSummarizeIntro} in the present setting.
	\begin{theorem} \label{thm.resume}
	 Let $X_1,\ldots,X_m$ be linearly independent smooth vector fields on $\RN$
	 \emph{(}with $N\geq 3$\emph{)}
	 satisfying the assumptions \emph{(I)-(II)} introduced in the Introduction. Moreover,
	 let $\LL = \sum_{j = 1}^mX_j^2$. Then, the following facts hold true:
	 \begin{itemize}
	  \item[\emph{(1)}] An open set $\Omega\subseteq\RN$ is a maximum principle set
	 for $\LL$ \emph{if and only if} its complement 
	 $\RN\setminus\Omega$ is $\LL$-large at infinity.
	 \item[\emph{(2)}] If $\Omega\subseteq\RN$
	 is an open set such that its complement $\RN\setminus\Omega$ is $p_\LL$-unbounded
	 (for some $p > 1$), then $\Omega$
	 is a maximum principle set for $\LL$.
	 \item[\emph{(3)}] If
	 $\Omega\subseteq\RN$ is an open set such that
	 its complement 
	 $\RN\setminus\Omega$ contains a $\Gamma$-cone,
	 then $\Omega$ is a maximum principle set for $\LL$.
	 \end{itemize}
	\end{theorem}
	We now proceed in this section by proving
	Proposition \ref{prop.delaconeIntro}
	stated in the In\-tro\-duc\-tion. To this end, we first establish the following
	result.
	\begin{proposition} \label{prop.deltaconeisGammacone}
	Let $F\subseteq\RN$ be a \emph{non-degenerate $\dela$-cone}, according
	to De\-fi\-ni\-tion \ref{defi.delacone}.
	Then $F$ is a $\Gamma$-cone.
	\end{proposition}
	\begin{proof}
	 According to Definition \ref{defi.Gammacone}, 
	 we have to prove the existence of a countable family
	 $\mathcal{F} = \{\Omega(z_n,R_n)\}_n$ 
	 such that $\Omega(z_n,R_n)\subseteq F$ for any $n\in\N$ and \medskip
	 
	 (a)\,\,$\|z_n\|\to\infty$ as $n\to\infty$; \medskip
	 
	 (b)\,\,$\liminf_{n\to\infty}R_n/\gamma_0(z_n) > 0$. \medskip
	 
	 \noindent To this end, we fix $z_0\in\mathrm{int}(F)\setminus\{0\}$ and we let
	 $R_0 > 0$ be such that $\Omega(z_0,R_0)\subseteq F$. 
	 Chosen a sequence $\{\lambda_n\}_n\subseteq(\lambda_0,\infty)$ diverging to
	 $\infty$ as $n\to\infty$, we define
	 $$\Omega_n = \Omega(z_n,R_n) := 
	 \Omega\big(\delta_{\lambda_n}(z_0), \lambda_n^{Q-2}\,R_0\big) \quad
	 \text{for every $n\in\N$}.$$
	 Since the fundamental solution $\Gamma$ of $\LL$
	 is jointly homogeneous of degree $2-Q$, 
	 it is straightforward to recognize that, for every $n\in\N$,
	 $$\Omega_n = \delta_{\lambda_n}\big(\Omega(z_0,R_0)\big);$$
	 hence, by property (ii) of $F$ we have
	 $\Omega_n\subseteq F$ for any $n\in\N$. Furthermore, we have
	 $\|z_n\| = \|\delta_{\lambda_n}(z_0)\|\to\infty$ as $n\to\infty$ and,
	 again by jointly homogeneity of $\Gamma$,
	 $$R_n/\gamma_0(z_n) = R_0/\gamma_0(z_0) > 0, \quad
	 \text{for every $n\in\N$}.$$ 
	 \noindent This shows that $\mathcal{F} := 
	 \{\Omega_n\}_{n}$ is a countable family
	 of $\Gamma$-balls (contained in $F$) satisfying (a) and (b), whence
	 $F$ is a $\Gamma$-cone. 
	\end{proof}
	\begin{remark} \label{rem.halfspaceCone}
	 Let $v\in\RN\setminus\{0\}$ be fixed
	 and let $h\in\R$. Then
	 the half-space
	 $$\Pi := \{x\in\RN: \langle x, v\rangle \geq h\}$$
	 contains a $\dela$-cone. Indeed, if we consider the subset of $\Pi$ defined by
	 $$\text{$C := \{x\in\Pi : x_iv_i\geq 0$ for any $i = 1,\ldots,N\}$},$$ 
	 it is very easy
	 to recognize that $\mathrm{int}(C) \neq \varnothing$; moreover,
	 $\dela(x)\in C$ {for every $x\in C$ and every
	 $\lambda > 1$}.
	 Hence, $C$ is a (non-degenerate) $\dela$-cone
	 contained in $\Pi$.
	\end{remark}
	By combining the above Proposition \ref{prop.deltaconeisGammacone} 
	with Theorem
   \ref{thm.GammaconeIntro}, we are able to provide a
	very simple proof of 
	Proposition \ref{prop.delaconeIntro}.
	\begin{proof} [Proof (of Proposition \ref{prop.delaconeIntro})]
	Let $F\subseteq\RN$ be as in the statement of the pro\-po\-si\-tion. 
	By assumption, there exists a non-degenerate
	$\dela$-cone $C\subseteq F$; on the other hand, by Proposition
	\ref{prop.deltaconeisGammacone}, $C$ is
	a $\Gamma$-cone (in the sense of Definition \ref{defi.Gammacone});
	as a consequence, from Theorem \ref{thm.GammaconeIntro} we infer
	the existence of a suitable $p > 1$ such that
	$F$ is $p_\LL$-unbounded. This ends the proof.
	\end{proof}
	The next Proposition \ref{prop.pboundqbounddx}, 
	which is the last result of the section,
	contains a useful characterization of the notion of
	$p_\LL$-boundedness in terms
	of the control distance $d_\mathrm{X}$ 
	(associated with the vector fields $X_1,\ldots,X_m$).
	\begin{proposition} \label{prop.pboundqbounddx}
	 Let $F\subseteq\RN$ be any
	 (non-void) set and let $p\in (1,\infty)$. 
	 Then, $F$ is $p_\LL$-bounded
	 \emph{if and only if} $F$ satisfies the following property: 
	 \emph{there exists
	 countable family $\mathcal{G} = \{B_\mathrm{X}(x_n,\rho_n)\}_{n\in J}$ such that
	 (setting $d_\mathrm{X}(x) := d_\mathrm{X}(0,x)$)}
	$$(\diamond)\qquad
	F\subseteq \bigcup_{n\in J}B_\mathrm{X}(x_n,\rho_n) \qquad\text{and}\qquad 
	\sum_{n \in J}\bigg(\frac{E(x_n,\rho_n)}{d^{Q-2}_\mathcal{X}(x_n)}\bigg)^p < \infty.$$
	\end{proposition}
	\begin{proof}
	 $(\Rightarrow)$\,\,Since, by assumption, $F$ is $p_\LL$-bounded,
	 it is possible to find a count\-able family
   $\mathcal{F} = \{\Omega(x_n,r_n)\}_{n\in J}$ 
   such that (see Definition \ref{def.pset}) 
   $$F\subseteq\bigcup_{n\in J}\Omega(x_n,r_n) \qquad\text{and}
   \qquad \sum_{n\in J}\big(\Gamma_0(x_n)\,r_n\big)^p < \infty;$$
   on the other hand, by the second inclusion in \eqref{eq.relationGammaballdball}, 
   for every $n\in J$ we have
   $$\Omega(x_n,r_n)\subseteq B_\mathcal{X}(x_n,\rho_n), \qquad\text{where $\rho_n = H(x_n,c_2 r_n)$}.$$
   Thus, if we consider the family 
   $\mathcal{G} = \{B_\mathrm{X}(x_n,\rho_n)\}_{n\in J}$, we see that
   $\mathcal{G}$ is a countable cover of $F$ such that
   (remind that $H(x,\cdot) = (E(x,\cdot))^{-1}$)
   $$\infty > \sum_{n\in J}\big(\Gamma_0(x_n)\,r_n\big)^p
   = \frac{1}{c_2^p}\,\sum_{n\in J}\big(\Gamma_0(x_n)\,
   E(x_n,\rho_n)\big)^p =: \big(\bigstar\big).$$
 	We now turn to give an estimate of $\Gamma_0(x_n)$ in terms
 	of $d_\mathcal{X}(x_n)$. To this end we observe that,
 	by the first inequality in \eqref{eq.globalSanchez}, we have
 	$$\Gamma_0(x_n) \geq \frac{1}{c_2}\,\frac{d^2_\mathcal{X}(x_n)}
 	{\Lambda\big(0,d_\mathcal{X}(x_n)\big)}
 	\quad \text{for every $n\in J$};$$
 	from this, taking into account 
 	\eqref{eq.Lambda0rrQ} in Remark 
 	\ref{rem.propertiesELambda}\,-\,$\mathrm{(d)_1}$, 
 	we obtain
 	\begin{equation*}
 	 \Gamma_0(x_n)\geq \frac{1}{\omega_Q c_2}\cdot\frac{1}{d_\mathrm{X}^{Q-2}(x_n)} \quad
 	 \text{for every $n\in J$}.
 	\end{equation*}
	By means of this last estimate we conclude that
	$$\infty > \big(\bigstar\big) \geq \bigg(\frac{1}{\omega_Q c_2^2}\bigg)^p\,
	\sum_{n \in J}\bigg(\frac{E(x_n,\rho_n)}{d^{Q-2}_\mathcal{X}(x_n)}\bigg)^p,$$
	and this proves that $\mathcal{G}$ satisfies $(\diamond)$. \medskip
	
	$(\Leftarrow)$\,\,Let $\mathcal{G} = \{B_\mathcal{X}(x_n,\rho_n)\}_{n\in J}$
	be a countable family of $d_\mathcal{X}$-balls satisfying $(\diamond)$. 
	By the first inclusion
	in
	\eqref{eq.relationGammaballdball}
	(and again by the fact that the maps $E(x,\cdot)$
	are $H(x,\cdot)$ are inverse to each other), it is easy to 
	recognize that
	$$B_\mathcal{X}(x_n,\rho_n) \subseteq \Omega(x_n,r_n), \qquad \text{where
	$r_n = c_2\cdot E(x_n,\rho_n)$};$$
	thus, if we consider the family
	$\mathcal{F} := \{\Omega(x_n,r_n)\}_{n\in J}$, we see that
	$\mathcal{F}$ is a countable cover of $F$ 
	(since the same is true of $\mathcal{G}$) and that
	$$\infty >
	\sum_{n \in J}\bigg(\frac{E(x_n,\rho_n)}{d^{Q-2}_\mathcal{X}(x_n)}\bigg)^p
	= \frac{1}{c_2^p}\,\sum_{n\in J}\bigg(\frac{r_n}{d^{Q-2}_\mathcal{X}(x_n)}\bigg)^p =: \big(\bigstar\big).$$
	On the other hand, by using the second inequality in \eqref{eq.globalSanchez}
	and by using again 
	\eqref{eq.Lambda0rrQ} in Remark \ref{rem.propertiesELambda}\,-\,$\mathrm{(d)_1}$,
	we derive that
	$$\Gamma_0(x_n) \leq c_2\,\frac{d^2_\mathcal{X}(x_n)}
 	{\Lambda\big(0,d_\mathcal{X}(x_n)\big)}
 	= \frac{c_2}{\omega_Q}\cdot\frac{1}{d^{Q-2}(x_n)} \quad \text{for every $n\in J$};$$
 	as a consequence, we obtain
 	$$\infty > \big(\bigstar\big)
 	\geq \bigg(\frac{\omega_Q}{c_2^2}\bigg)^p\,\sum_{n\in J}\big(\Gamma_0(x_n)\,r_n\big)^p$$
 	and this proves that $F$ is $p_\LL$-bounded. 
	\end{proof}
	\begin{remark} \label{rem.finalLiegroups}
	As a final remark we observe that, in the particular case when
	$X_1,\ldots,X_m$ are \emph{Lie generators}
	of the Lie algebra of some homogeneous Carnot group
	on $\RN$ (see \cite[Chapter 1]{BLUlibro} for the relevant definitions), we have
	$$E(x,r) = \omega_Q\,r^{Q-2} \qquad\text{for every $x\in\RN$ and every $r > 0$};$$ 
	as a consequence, a set $F\subseteq\RN$ is $p_\LL$-bounded (for some $p > 1$)
	\emph{if and only if} there exists a countable family
	$\mathcal{G} = \{B_\mathrm{X}(x_n,\rho_n)\}_{n\in J}$ such that
	$$F\subseteq \bigcup_{n\in J}B_\mathrm{X}(x_n,\rho_n) \qquad\text{and}\qquad 
	\sum_{n \in J}\bigg(\frac{\rho_n}{d_\mathcal{X}(x_n)}\bigg)^{p(Q-2)} < \infty.$$
	Due to this fact, the results presented in this paper comprehend and generalize 
	that contained in \cite{BonfiLancoMP} (see also \cite[Chapter 10]{BLUlibro}).
  \end{remark}
  \appendix
  \section{Appendix: some results of Potential Theory} \label{sec:Appendix}
  The main aim of this brief appendix is to collect some notions and
  results, coming from Potential Theory, needed to prove 
  Lemma \ref{lem.gluingsubHb} in Section
  \ref{sec:thinsets}.
  In our exposition we mainly follow the book by Br{e}lot \cite{Brelot}, to which
  we refer for a detailed treatment of these topics (and for the proof
  of all the results we are going to state); we also highlight
  the very classical references \cite{Bauer,ConstCornea}. \medskip
  
  Throughout the sequel, we denote by $\LL$ a fixed linear PDO as in \eqref{eq:PDOformintro}
  and satisfying the structural assumptions (H1)-to-(H3);
  moreover, we tacitly inherit all the notations
  introduced in the previous sections.
  \subsection*{The $\LL$-harmonic space}
  We begin with the following simple observation: if $\tau_\mathcal{E}$
  denotes the usual Euclidean topology on $\RN$, then the assignment 
  \begin{equation} \label{eq.mapsheaf}
  \tau_\mathcal{E}\ni\Omega\mapsto \mathcal{L}(\Omega)
  = \{u\in C^2(\Omega,\R):\LL u = 0\,\,\text{in $\Omega$}\},
  \end{equation}
  is a sheaf of functions on $\RN$. More precisely, we have
  \medskip

   (i)\,\,for any $\Omega\in\tau_\mathcal{E}$, 
   $\mathcal{L}(\Omega)$ is a linear subspace
   of $C(\Omega,\R)$; \vspace*{0.35cm}

   (ii)\,\,if $\Omega_1\subseteq\Omega_2$ are open subsets of $\RN$
    and if $u\in\mathcal{L}(\Omega_2)$, then 
    $u_{\big|\Omega_1}\in\mathcal{L}(\Omega_1)$; \medskip

    (iii)\,\,if $\{\Omega_i\}_{i\in I}\subseteq\tau_\mathcal{E}$ 
    and if $u:\Omega:=\bigcup_{i\in I}\Omega_i\to\R$, then
    $$\bigg(\text{$u_{\big|\Omega_i}\in\mathcal{L}(\Omega_i)$ for all $i\in I$}\bigg)\,\,
    \Longrightarrow \,\, u\in\mathcal{L}(\Omega).$$  
  \medskip
  
  \noindent Let now $\Omega\subseteq\RN$ be an open set. We say that $\Omega$ is 
  \emph{$\LL$-regular} if
   \begin{itemize}
    \item[(i)] $\overline{\Omega}$ is compact;
    \item[(ii)] for every continuous function $f:\de\Omega\to\R$ there exists a unique
    $\mathcal{L}$-harmonic function in $\Omega$, denoted by
    $H^\Omega_f$, such that
    $$\lim_{x\to\xi}H^\Omega_f(x) = f(\xi), \qquad \text{for all $\xi\in\de\Omega$};$$
    \item[(iii)] if $f\geq 0$ on $\de\Omega$, then $H^V_f\geq 0$ in $\Omega$.
   \end{itemize}
   It is very easy to see that, if $\Omega$ is $\LL$-regular, the map
   $$T:C(\de\Omega,\R)\longto\R, \qquad T(f) := H^\Omega_f(x)$$
   is linear and positive; since $\de\Omega$ is compact,
   the Riesz Representation Theorem (see, e.g., \cite{Rudin}) provides
   us with a unique Radon measure $\mu^\Omega_x$ on $\Omega$ such that
   $$H^\Omega_f(x) = \int_{\de\Omega}f(y)\,\d\mu^\Omega_x(y).$$
   The measure $\mu^\Omega_x$ is called the \emph{$\mathcal{L}$-harmonic measure}
   related to $\Omega$ and $x$.  \medskip

   As a consequence of some results proved in \cite{BBB}
   (see, precisely, Lemma 1.7 and Theorem 1.10), we see that the following
   facts hold true for our PDO $\LL$:
   \begin{itemize}
    \item[(a)] there exists a (countable) basis for the Euclidean topology of $\RN$
    consisting of connected $\LL$-regular open sets;
    \item[(b)] for every connected open set $\Omega\subseteq \RN$
     and every compact set $K\subseteq \Omega$, there exists a constant
     $C = C(\Omega,K) \geq 1$ such that
     $$\sup_K u \leq C\,\inf_K u,$$
     for every non-negative harmonic function $u$ in $\Omega$.
   \end{itemize}
	On the other hand, the validity of (a) and (b) easily implies the following results.
	\begin{itemize}
	 \item[(1)] Let $\Omega\subseteq\RN$ be an open
	 set (not necessarily $\LL$-regular) and let $u\in C(\Omega,\R)$.
	 Then the function $u$ is $\LL$-harmonic in $\Omega$ if and only if
	 $$u(x) = \int_{\de V}u(z)\,\d\mu^V_x(z),$$
	 for every $\LL$-regular open set
	  $V\subseteq\overline{V}\subseteq\Omega$ and every $x\in V$.
	 \item[(2)] If $\Omega\subseteq \RN$ is open and connected
     and $\{u_n\}_n\subseteq\mathcal{L}(\Omega)$ is monotone increasing, then
     either $\sup_n u_n \equiv \infty$ on 
     $\Omega$ or it is a $\mathcal{L}$-harmonic function in
     $\Omega$.
	\end{itemize}
    Gathering together all these facts, we recognize that the map 
    defined in \eqref{eq.mapsheaf}
    sa\-tisfies Axioms 1-to-3 in \cite{Brelot}; hence,
    it endows $\RN$ with the structure of a harmonic sheaf, which
    is usually referred to as the \emph{$\LL$-harmonic space}.
    \subsection*{$\LL$-subharmonic functions}
    Let $\Omega\subseteq\RN$ be an open set
    and let $u:\Omega\to [-\infty,\infty)$ be u.s.c.\,on $\Omega$.
    As already said in the Introduction,
    the function $u$ is \emph{$\LL$-subharmonic} in $\Omega$ if
    \begin{itemize}
  	\item[(i)] $\{x\in\Omega: u(x) > -\infty\}$ is dense in $\Omega$;
  
  	\item[(ii)] for every bounded open 
  	set $V\subseteq\overline{V}\subseteq\Omega$ and for every
  function $h$ $\LL$-harmonic in $V$ and continuous up to $\de V$ such that
  $u_{\big| \de V}\leq h_{\big|\de V}$, one has $u\leq h$ in $V$.
  \end{itemize}
  \begin{remark} \label{rem.defisubHorig}
   Let $\Omega\subseteq\RN$ be open and let $u\in\subH(\Omega)$. Moreover,
   let $V$ be a $\LL$-regular open set such that $\overline{V}\subseteq \Omega$.
   If $\varphi\in C(\de V,\R)$ is any continuous function
   satisfying $u\leq \varphi$
  on $\de V$,
   by (ii) we have
  $$u(x) \leq H^V_\varphi(x) = \int_{\de V}\varphi(z)\,\d\mu^V_x(z)
  \quad \text{for every $x\in V$}.$$
   From this, due to the arbitrariness of $\varphi$, we obtain
	\begin{equation} \label{eq.defusubHoriginal}   
   u(x) \leq \int_{\de V}u(z)\,\d\mu^V_x(z)
  \quad \text{for every $x\in V$}.
  \end{equation}
  \end{remark}
  For $\mathcal{L}$-subharmonic functions, we have the following
   minimum principles.
   \begin{theorem} \label{thm.minimumsupH}
    Let $\Omega\subseteq \RN$ be open
    and let $u\in\subH(\Omega)$. The following facts hold:
    \begin{itemize}
     \item[\emph{1.}] if $\Omega$ is connected and $u\leq 0$
     on $\Omega$, then either $u \equiv 0$ or $u < 0$;
     \item[\emph{2.}] if $\Omega$ is \emph{bounded},
      $\liminf_{x\to\xi}u(x)\leq 0$ for any $\xi\in\de\Omega$
     and there exists a $\mathcal{L}$-harmonic function $h$ such that $\inf_\Omega h > 0$,
     then $u\leq 0$ on $\Omega$.
    \end{itemize}
   \end{theorem}
   Theorem \ref{thm.minimumsupH} allows us to prove
   that condition \eqref{eq.defusubHoriginal} in Remark \ref{rem.defisubHorig}
   actually characterizes, even in a suitable local form,
   $\mathcal{L}$-subharmonicity.
   \begin{proposition} \label{prop.localcritsuph}
    Let $\Omega\subseteq \RN$ be open and let $u:\Omega\to(-\infty,\infty]$
    be a u.s.c.\,fun\-ction such that the set $D := \{x\in\Omega:u(x)>-\infty\}$
    is dense in $\Omega$.

    Then the following conditions are equivalent:
    \begin{itemize}
    \item[\emph{(a)}] $u\in\subH(\Omega)$;
	\item[\emph{(b)}] for every $\LL$-regular open set
	$V\subseteq\overline{V}\subseteq\Omega$
	and for every function $\varphi\in C(\de V,\R)$ satisfying $u\leq \varphi$ on $\de V$, one has    
	$$u(x) \leq \int_{\de V}u(z)\,\d\mu^V_x(z)
  \quad \text{for every $x\in V$}.$$
    \item[\emph{(c)}] for every $x_0\in\Omega$ there exists a basis
    $\mathcal{B}_u$ (possibly depending on $u$)
    of $\mathcal{L}$-regular open neighborhoods of $x_0$ such that,
    for any $V\in\mathcal{B}_u$, one has
    $$u(x_0)\leq \int_{\de V}u(y)\,\d\mu^V_{x_0}(y).$$
    \end{itemize}
   \end{proposition}
   \subsection*{Proof of Lemma \ref{lem.gluingsubHb}}
	Thanks to Proposition \ref{prop.localcritsuph},
	we are finally in a position to prove Lemma \ref{lem.gluingsubHb}. For the
	sake of clarity, we re-write here its statement.
	\begin{lemma} \label{lem.gluingsubHb_bisApp}
   Let $\Omega\subseteq\RN$ be open and let
   $u\in\subH(\Omega)$ be such that
   \begin{equation} \label{eq.conditionlemma_bisApp}
    \limsup_{x\to \xi}u(x) \leq 0 \quad \text{for every $\xi\in\de\Omega$}.
   \end{equation}
   Then the function $v:\RN\to[-\infty,\infty)$ defined by
   $$v(x) = \begin{cases}
    \max\{u(x),0\}, & \text{if $x\in\Omega$}, \\[0.1cm]
    0, & \text{if $x\in\RN\setminus\Omega$},
   \end{cases}
   $$
   is $\LL$-subharmonic in $\RN$.
  \end{lemma}
  \begin{proof}
   First of all, 
   condition \eqref{eq.conditionlemma_bisApp} ensures that $v$ is u.s.c.\,on $\Omega$; moreover,
   by the very definition of $v$, we have
   $v \geq 0 > -\infty$ on the whole of $\RN$.

   To prove that $v\in\subH(\RN)$ we show that, for every $x_0\in\RN$, there exists a
   basis $\mathcal{B}(x_0)$ of $\LL$-regular open neighborhoods of $x_0$ such that
   (see Proposition \ref{prop.localcritsuph})
   $$v(x_0) \leq \int_{\de V}v(y)\,\d\mu^V_{x_0}(y), \quad \text{for every $V\in\mathcal{B}(x_0)$}.$$
   If $x_0\in\Omega$, we can choose as $\mathcal{B}(x_0)$ the family
   of the $\LL$-regular open neighborhoods of $x_0$ with closure contained in $\Omega$: indeed,
   since $f := \max\{u,0\}$ is $\LL$-subharmonic in $\Omega$
   (as the same is true of both $u$ and $0$), we have
   $$v(x_0) = f(x_0) \leq \int_{\de V}f(y)\,\d\mu^V_{x_0}(y) = \int_{\de V}u(y)\,\d\mu^V_{x_0}(y)$$
   for every $\LL$-regular open neighborhood $V$
   of $x_0$ with $\overline{V}\subseteq\Omega$. If, instead, $x_0\notin\Omega$,
   we can choose as $\mathcal{B}(x_0)$ the family of all $\LL$-regular open neighborhoods
   of $x_0$: indeed, since $v\geq 0$ on the whole of $\RN$ (by definition), we have
   $$v(x_0) = 0 \leq \int_{\de V}u(y)\,\d\mu^V_{x_0}(y)$$
   for every $\LL$-regular open neighborhood $V$ of $x_0$. This ends the proof.   
  \end{proof}
   
\end{document}